\documentclass{article}
\usepackage{amsmath}
\usepackage{amsfonts}
\usepackage{amsthm}
\usepackage{enumerate}

\newtheorem{thm}{Theorem}[section]
\newtheorem{prop}[thm]{Proposition}
\newtheorem{cor}[thm]{Corollary}

\newtheorem{lem}[thm]{Lemma}

\theoremstyle{remark}

\begin{document}

\title{Evaluating Theta Derivatives with Rational Characteristics}

\author{Shaul Zemel}

\maketitle

\begin{center}
\textit{Dedicated to the memory of my grandfather, Isaac Zemel}
\end{center}

\begin{abstract}
We evaluate 35 different theta derivatives with rational characteristics using the method of Matsuda, in which the logarithmic derivatives at $z=0$ of the triple product expansions of the associated theta functions become linear combinations of known theta series of weight 1.
\end{abstract}

\section*{Introduction}

The \emph{Jacobi theta function} $\theta(\tau,z)=\sum_{n=-\infty}^{\infty}e^{\pi in^{2}\tau+2\pi inz}$, as well as the associated \emph{theta constant} $\theta(\tau,0)$, are well-known for many decades to be fundamental objects in number theory. One reason is that the theta constant is a modular form of weight $\frac{1}{2}$ with respect to a congruence subgroup of low level (in fact, \cite{[SS]} shows that variants of this theta function generate all the modular forms of weight $\frac{1}{2}$ with respect to every congruence subgroup), while $\theta(\tau,z)$ is the first example of a Jacobi form (see \cite{[EZ]} for more on this subject). These functions also have applications in combinatorics, enumerative number theory (e.g., the theory of partitions), the theory of lattices (which yields generalizations of these theta functions), differential equations (the Jacobi theta function is a fundamental solution to the heat equation), and many other branches in mathematics.

One considers more general theta functions, involving \emph{characteristics}. The effect of adding characteristics to the theta function boils down to  translations of $z$ and the summation index $n$ by some real number. There are many relations between the theta functions with various characteristics (e.g., the characteristics are essentially well-defined modulo $2\mathbb{Z}$), but the substitution $z=0$ yields many different theta constants. The theta constants with integral characteristics are very well-known modular forms of weight $\frac{1}{2}$, and the investigation of their properties goes back to the times of Jacobi and Euler. However, as integral characteristics yield only three non-zero theta constants, the number of possible constructions from them remains limited. Taking the derivative of a theta function with respect to $z$ at $z=0$ (such expressions are called \emph{theta derivatives}) produces a modular form of weight $\frac{3}{2}$, but with integral characteristics one obtains only one non-zero theta derivative. Moreover, the Jacobi derivative formula states that this additional function is not new, but is (up to a multiplicative constant) the product of the previous three theta constants.

It is a simple observation that by allowing the characteristics to become rational, the theta constants become modular forms (again of weight $\frac{1}{2}$) with respect to congruence subgroups of higher level (in fact, multiplying the argument $\tau$ by appropriate integers and taking appropriate linear combinations yields all the functions considered in \cite{[SS]} in this way). A similar assertion, with the weight being $\frac{3}{2}$, holds for the theta derivatives. \cite{[Mu]} posed the question, appearing as Question II on page 117 of that reference, whether all the theta derivatives of rational characteristics could be presented as polynomials of degree 3 in theta constants with rational characteristics, analogously to the Jacobi derivative formula. A positive answer to that question will, in view of \cite{[SS]}, simplify significantly the analysis of the rings of modular forms of half-integral weight with respect to these congruence subgroups. A more precise analysis of which theta function (or combination of such functions) is modular of which level may be found in \cite{[B1]}.

Recently \cite{[M1]} developed a method to find such expressions, in which one of the multipliers has the same characteristics as the theta derivative in question and the others have integral characteristics (but with different arguments $a\tau$ for positive rational $a$). The idea is to evaluate the logarithmic derivative of the associated theta function at $z=0$, when this function is expressed via the \emph{Jacobi triple product identity} as an infinite product of simple expressions. When the resulting series can be recognized as a known one, we obtain the desired expression for the theta derivative. We remark that \cite{[M1]} goes on to establish, by different means, other expressions for the same theta derivatives that are \emph{rational functions} of the theta constants. The sequels \cite{[M2]} and \cite{[M3]} go further in that direction, and obtain formulae comparing theta derivatives with other rational characteristics to rational functions in theta constants. However, as we are interested only in polynomial expressions in this paper, we stick to the first method from \cite{[M1]} in this paper. We remark that formulae like the Jacobi triple product identity, relating infinite series to infinite products, are investigated in \cite{[Z]}, \cite{[C]}, and others. The latter reference is a very recent and comprehensive treatise of various types of theta functions and the relations among them.

Now, \cite{[M1]} used this method in order to evaluate three theta derivatives, namely those in which the upper characteristic is 1 and the lower one is $\frac{1}{2}$, $\frac{1}{4}$, or $\frac{3}{4}$. For this explicit expressions for two theta series (which are essentially the Hecke theta series arising from the imaginary quadratic fields $\mathbb{Q}(\sqrt{i})$ and $\mathbb{Q}(\sqrt{-2})$) are used. We show that using only these tools, and perhaps some simple algebraic manipulations, 19 further theta derivatives (all those in which both the characteristics are in $\frac{1}{4}\mathbb{Z}$, but at least one of them is in $\frac{1}{2}\mathbb{Z}$) can be evaluated. Moreover, when one recognizes the theta series associated with $\mathbb{Q}(\sqrt{-3})$, we can evaluate 16 additional theta derivatives (those whose characteristics lie in $\frac{1}{3}\mathbb{Z}$). We remark that for other imaginary quadratic fields (also those with class number 1) additional technical difficulties arise. We leave the analysis of theta derivatives using such theta series for future investigation.

The Jacobi derivative formula also provides us with the Fourier expansion of the third power of the Dedekind eta function. \cite{[M1]} also uses his results (for the characteristic with $\frac{1}{2}$) for obtaining the Fourier expansion of another eta quotient of weight $\frac{3}{2}$. We indicate how some additional results in this direction as well (many of which are parallel to the main results of \cite{[LO]}), though in most of the cases our method yields the Fourier expansion of a linear combination of eta quotients, rather than of a single eta quotient. Some relations between eta quotients also follow from this argument.

\smallskip

This paper is divided into 3 sections. Section \ref{ThetaGen} describes the general observations about theta functions that are required for applying the argument (in fact, this presentation makes the paper almost fully self-contained, up to a proof for the Jacobi triple product identity itself). Section \ref{Charin1/4Z} presents the results of \cite{[M1]} again, and shows how the argument generalizes for the additional 19 theta derivatives in which the characteristics are from $\frac{1}{4}\mathbb{Z}$. Finally, Section \ref{Charin1/3Z} proves the results for characteristics from $\frac{1}{3}\mathbb{Z}$.

I would like to thank H. Farkas for many interesting discussions on this subject, as well as K. Matusda for sharing the details of the manuscript \cite{[M2]} with me (including the reference of the original question to \cite{[Mu]}).

\section{Theta Functions and Theta Derivatives \label{ThetaGen}}

Given two real numbers $\varepsilon$ and $\delta$, the \emph{theta function with characteristics $\big[{\varepsilon \atop \delta}\big]$} is a function of two variables, $\tau$ from the \emph{upper half-plane} $\mathcal{H}$ (i.e., a complex number with positive imaginary part), and $z\in\mathbb{C}$. It is defined by the infinite series expansion
\begin{equation}
\theta\big[{\textstyle{\varepsilon \atop \delta}}\big](\tau,z)=\sum_{n=-\infty}^{\infty}\mathbf{e}\Big[\big(n+\tfrac{\varepsilon}{2}\big)^{2}\tfrac{\tau}{2}+
\big(n+\tfrac{\varepsilon}{2}\big)\big(z+\tfrac{\delta}{2}\big)\Big], \label{defser}
\end{equation}
where $\mathbf{e}(u)$ stands for $e^{2\pi iu}$ for any complex number $u$. In addition we shall be using the classical notation $q=\mathbf{e}(\tau)$ and $\zeta=\mathbf{e}(z)$ (note that some authors, including \cite{[M1]} on which our work is based, use $q$ for the square root of our $q$), while the primitive $N$th root of unity $\mathbf{e}\big(\frac{1}{N}\big)$ will be denoted by $\zeta_{N}$ (no confusion with $\zeta$ should arise). The value of this function at $z=0$ and the value of the derivative with respect to $z$ at $z=0$ are called the \emph{theta constant} and \emph{theta derivative} with characteristics $\big[{\varepsilon \atop \delta}\big]$ respectively. We are interested in explicit expressions for the derivatives in terms of the constants, i.e., writing theta derivatives with rational characteristics in terms of theta constants. The Fourier expansions of the results will be clear from our arguments, while some of the final expression will also be phrased in terms of the Dedekind eta function \[\eta(\tau)=q^{1/24}\prod_{n=1}^{\infty}(1-q^{n})\] (the eta function can be viewed as a theta constant itself---see Proposition \ref{charover3} below). As with rational $\varepsilon$ and $\delta$ these theta constants become, as functions of $\tau$, modular forms of weight $\frac{1}{2}$ on some congruence subgroup of $SL_{2}(\mathbb{Z})$ (and the same holds for $\eta$), and the derivatives have weight $\frac{3}{2}$, one naturally expects to be able to express the derivatives as linear combinations of products of three theta constants (or of eta quotients of total weight $\frac{3}{2}$).

First, the theta functions have some well-known properties, most of which follow by a simple summation index change (see, e.g., Section 1 in Chapter 2 of \cite{[FK]}). Explicitly, we have the pseudo-periodicity property in $z$ with respect to the lattice determined by $\tau$, namely
\begin{equation}
\theta\big[{\textstyle{\varepsilon \atop \delta}}\big](\tau,z+a\tau+b)=\mathbf{e}\big(\tfrac{b\varepsilon-a\delta}{2}\big)\zeta^{-a}q^{-a^{2}/2}\theta\big[{\textstyle{\varepsilon \atop \delta}}\big](\tau,z), \label{perz}
\end{equation}
as well as the pseudo-2-periodicity in the characteristics,
\begin{equation}
\theta\big[{\textstyle{\varepsilon+2a \atop \delta+2b}}\big](\tau,z)=\mathbf{e}\big(\tfrac{b\varepsilon}{2}\big)\theta\big[{\textstyle{\varepsilon \atop \delta}}\big](\tau,z), \label{perchar}
\end{equation}
both holding for integral $a$ and $b$, real $\varepsilon$ and $\delta$, and every $\tau\in\mathcal{H}$ and $z\in\mathbb{C}$. Allowing arbitrary real $a$ and $b$, all the theta functions are related to one another via the relation
\begin{equation}
\theta\big[{\textstyle{\varepsilon+a \atop \delta+b}}\big](\tau,z)=\mathbf{e}\big(\tfrac{a(b+\delta)}{4}\big)\zeta^{a/2}q^{a^{2}/8}\theta\big[{\textstyle{\varepsilon \atop \delta}}\big]\big(\tau,z+\tfrac{a\tau+b}{2}\big) \label{charrel}
\end{equation}
(with Equation \eqref{perchar} following from Equations \eqref{perz} and \eqref{charrel}), and we also have the parity relation
\begin{equation}
\theta\big[{\textstyle{-\varepsilon \atop -\delta}}\big](\tau,z)=\theta\big[{\textstyle{\varepsilon \atop \delta}}\big](\tau,-z),\quad\mathrm{whence\ also}\quad\theta\big[{\textstyle{-\varepsilon \atop -\delta}}\big]'(\tau,z)=-\theta\big[{\textstyle{\varepsilon \atop \delta}}\big]'(\tau,-z). \label{parity}
\end{equation}
Here and throughout, the prime denotes differentiation with respect to $z$ only. Now, Equation \eqref{perchar} allows us to restrict attention to $\varepsilon$ and $\delta$ in $[0,2)$ (or equivalently in $(-1,1]$). Integration over the boundary of a fundamental parallelogram with respect to $\tau$ shows, using equation \eqref{perz}, that each theta function with characteristics has precisely one zero (in $z$) in each such parallelogram. Using Equations \eqref{parity} and \eqref{perz} we determine this zero for the case $\varepsilon=\delta=1$ to be $z=0$ (and its $\tau$-lattice translates), and Equation \eqref{charrel} then yields the following lemma.
\begin{lem}
If $\varepsilon$ and $\delta$ are in $(-1,1]$ then the unique zero of $\theta\big[{\varepsilon \atop \delta}\big](\tau,z)$ that is of the form $z=c\tau+d$ with $c$ and $d$ in $[0,1)$ is the point $z=\frac{1-\varepsilon}{2}\tau+\frac{1-\delta}{2}$. \label{zero}
\end{lem}
Moreover, by Equation \eqref{parity} we may further restrict our attention to theta functions (and constants, and derivatives) in which the characteristics $\big[{\varepsilon \atop \delta}\big]$ satisfy  $0\leq\varepsilon\leq1$ and $0\leq\delta<2$, and if $\varepsilon\in\{0,1\}$ then we can further consider only $0\leq\delta\leq1$.

\smallskip

The main idea of the proofs of all the following identities, following \cite{[M1]}, is the use of the \emph{Jacobi triple product identity}. This identity is given by the equality
\[\sum_{n=-\infty}^{\infty}x^{n^{2}/2}y^{n}=\prod_{n=1}^{\infty}(1-x^{n})(1+x^{n}y)\big(1+\tfrac{x^{n}}{y}\big),\] holding either as formal power series, or as convergent complex expressions in case $|x|<1$ (and $y\neq0$, of course). For a simple proof of it see \cite{[A]} or Section 1.1 of \cite{[K]}. Taking $x=q$ and $y=\mathbf{e}\big(\frac{\delta}{2}\big)q^{\varepsilon/2}\zeta$ in the Jacobi triple product identity, we find that $\theta\big[{\varepsilon \atop \delta}\big](\tau,z)$ (related to $\theta\big[{0 \atop 0}\big]$ as in Equation \eqref{charrel}) can be expanded as
\begin{equation}
\mathbf{e}\big(\tfrac{\varepsilon\delta}{4}\big)q^{\varepsilon^{2}/8}\zeta^{\varepsilon/2}\prod_{n=1}^{\infty}(1-q^{n})
\Big(1+\mathbf{e}\big(\tfrac{\delta}{2}\big)q^{n-\frac{1-\varepsilon}{2}}\zeta\Big)\Big(1+\mathbf{e}\big(-\tfrac{\delta}{2}\big)q^{n-\frac{1+\varepsilon}{2}}/\zeta\Big). \label{tripprod}
\end{equation}
The vanishing from Lemma \ref{zero} (in particular the vanishing of $\theta\big[{1 \atop 1}\big](\tau,0)$) can be easily seen from Equation \eqref{tripprod} as well, in which it occurs only at the rightmost term with $n=1$. In fact, all the properties appearing in Equations \eqref{perz}, \eqref{perchar}, \eqref{charrel}, and \eqref{parity} can be established by appropriately manipulating Equation \eqref{tripprod}.

\smallskip

The first relations that one can deduce from Equation \eqref{tripprod} are expressions for the theta constants with integral characteristics in terms of the Dedekind eta function.
\begin{prop}
We have the following expressions for the theta constants with integral characteristics: \[\theta\big[{\textstyle{1 \atop 0}}\big](\tau,0)=2\frac{\eta^{2}(2\tau)}{\eta(\tau)},\ \theta\big[{\textstyle{0 \atop 1}}\big](\tau,0)=\frac{\eta^{2}\big(\frac{\tau}{2}\big)}{\eta(\tau)},\mathrm{\ and\ }\theta\big[{\textstyle{0 \atop 0}}\big](\tau,0)=\frac{\eta^{5}(\tau)}{\eta^{2}(2\tau)\eta^{2}\big(\frac{\tau}{2}\big)}.\] \label{intchar}
\end{prop}

\begin{proof}
For $\theta\big[{1 \atop 0}\big](\tau,0)$ the last term with $n=1$ in Equation \eqref{tripprod} yields the constant 2, and the rest gives the product $q^{1/8}\prod_{i=1}^{n}(1-q^{n})(1+q^{n})^{2}$. By writing $1+q^{n}$ as $\frac{1-q^{2n}}{1-q^{n}}$ we obtain the required expression. Next, Equation \eqref{tripprod} presents $\theta\big[{0 \atop 1}\big](\tau,0)$ as just $\prod_{i=1}^{n}(1-q^{n})(1-q^{n-1/2})^{2}$, and the fact that $\prod_{i=1}^{n}(1-q^{n-1/2})$ is the quotient $\prod_{i=1}^{n}(1-q^{n/2})/\prod_{i=1}^{n}(1-q^{n})$ yields the desired expression. Finally, Equation \eqref{tripprod} implies that $\theta\big[{0 \atop 0}\big](\tau,0)$ equals $\prod_{i=1}^{n}(1-q^{n})(1+q^{n-1/2})^{2}$. We replace each $1+q^{n+1/2}$ by $\frac{1-q^{2n-1}}{1-q^{n-1/2}}$, and after writing $\prod_{i=1}^{n}(1-q^{n-1/2})$ as the quotient from above as well as $\prod_{i=1}^{n}(1-q^{2n-1})$ as $\prod_{i=1}^{n}(1-q^{n})/\prod_{i=1}^{n}(1-q^{2n})$ (by the exact same argument), the asserted expression follows (the external powers of $q$ combine to the correct factors in all three cases). This proves the proposition
\end{proof}
Proposition \ref{intchar} is, up to rescaling of the variable $\tau$, a restatement of the first and third equalities in Theorem 1.1 of \cite{[LO]} and the first one in Theorem 1.2 there, or equivalently of Equations (8.5), (8.7), and (8.8) inside Theorem 8.1 of \cite{[K]}. Similar expressions for some theta constants with rational characteristics in terms of eta quotients (including ones relating single eta functions to particular theta constants), which are equivalent to other results from these references, are given in Propositions \ref{halfchar} and \ref{charover3} below.

\smallskip

We have seen that $\theta\big[{1 \atop 1}\big](\tau,0)$ vanishes, and the same assertion holds for the three theta derivatives $\theta\big[{0 \atop 0}\big]'(\tau,0)$, $\theta\big[{0 \atop 1}\big]'(\tau,0)$, and $\theta\big[{1 \atop 0}\big]'(\tau,0)$. The remaining integral theta derivative is given in terms of the \emph{Jacobi derivative formula}, that reads
\begin{equation}
\theta\big[{\textstyle{1 \atop 1}}\big]'(\tau,0)=-\pi\theta\big[{\textstyle{0 \atop 0}}\big](\tau,0)\theta\big[{\textstyle{0 \atop 1}}\big](\tau,0)\theta\big[{\textstyle{1 \atop 0}}\big](\tau,0)=-2\pi\eta^{3}(\tau). \label{clasder}
\end{equation}
By replacing the theta derivative on the left hand side of Equation \eqref{clasder} by the derivative of the series from Equation \eqref{defser} at $z=0$ and using the obvious symmetry between $n$ and $-1-n$ we find the well-known Fourier expansion
\begin{equation}
\eta^{3}(\tau)=\sum_{n=1}^{\infty}(-1)^{n}(2n+1)q^{(2n+1)^{2}/8}=\sum_{n=1}^{\infty}n\big(\tfrac{-4}{n}\big)q^{n^{2}/8}, \label{eta3exp}
\end{equation}
appearing (again up to variable rescaling) in part 2 of Theorem 1.1 of \cite{[LO]}, in Equation (1.7) in Corollary 1.4 of \cite{[K]},  as well as in Equation (8.15) in Theorem 8.5 of that reference (recall that dividing Equation \eqref{eta3exp} by $q^{1/8}$ replaces $\eta^{3}(\tau)$ by $\prod_{n=1}^{\infty}(1-q^{n})^{3}$ and the power $\frac{(2n+1)^{2}}{8}$ of $q$ by $\frac{n(n+1)}{2}$, thus reproducing the classical identity of Jacobi from, e.g., Equation (1.7) inside Corollary 1.4 of \cite{[K]}). A simple proof of Equation \eqref{clasder}, which we shall indicate in the following paragraph, can be given using the Jacobi triple product identity. \cite{[M1]} derived similar expressions, again using the Jacobi triple product identity, for $\theta\big[{1 \atop 1/2}\big]'(\tau,0)$, $\theta\big[{1 \atop 1/4}\big]'(\tau,0)$, and $\theta\big[{1 \atop 3/4}\big]'(\tau,0)$ (see his Theorems 1 and 2---the later results of that reference use additional tools, which we shall not require in our investigation).

The case of the classical Jacobi derivative formula is simple: The expansion of $\theta\big[{1 \atop 1}\big](\tau,z)$ includes the multiplier $1-\frac{1}{\zeta}$ from the rightmost term with $n=1$ in Equation \eqref{tripprod}. Hence the derivative at $z=0$ is just $2\pi i$ times the remaining expression evaluated at $z=0$, and this substitution gives $iq^{1/8}\prod_{i=1}^{n}(1-q^{n})^{3}$. This yields the rightmost expression in Equation \eqref{clasder}, and the relation with the product of the three theta constants is immediate from Proposition \ref{intchar}. However, for other characteristics the theta constant does not vanish, and the idea of \cite{[M1]} is to take the \emph{logarithmic} derivative of the expressions coming from the theta function in question via Jacobi triple product identity. Hence the classical, integral case is a bit misleading as for the argument required for all the other cases.

\smallskip

We therefore evaluate the logarithmic derivative of the expression from Equation \eqref{tripprod} at $z=0$, under the restrictions on $\varepsilon$ and $\delta$ from above (but also excluding the case $\varepsilon=\delta=1$), which yields \[\frac{\theta\big[{\varepsilon \atop \delta}\big]'(\tau,0)}{\theta\big[{\varepsilon \atop \delta}\big](\tau,0)}=2\pi i\Bigg[\frac{\varepsilon}{2}+\sum_{n=1}^{\infty}\Bigg(\frac{\mathbf{e}\big(\frac{\delta}{2}\big)q^{n-\frac{1-\varepsilon}{2}}}{1+\mathbf{e}\big(\frac{\delta}{2}\big)
q^{n-\frac{1-\varepsilon}{2}}}-\frac{\mathbf{e}\big(-\frac{\delta}{2}\big)q^{n-\frac{1+\varepsilon}{2}}}{1+\mathbf{e}\big(-\frac{\delta}{2}\big)q^{n-\frac{1+\varepsilon}{2}}}\Bigg)\Bigg].\]
Recalling that $|q|<1$ and that $0\leq\varepsilon\leq1$, we may expand every denominator as a geometric series, except when $\varepsilon=1$, where the rightmost term with $n=1$ is just a constant. By setting $\beta=-\mathbf{e}\big(-\frac{\delta}{2}\big)$ (which is a root of unity when $\delta\in\mathbb{Q}$) and changing the order of the terms for later convenience, the latter expression for the theta derivative becomes
\begin{equation}
2\pi i\bigg[\frac{1+\beta}{2(1-\beta)}+\sum_{n=1}^{\infty}\sum_{l=1}^{\infty}(\beta^{l}-\overline{\beta}^{l})q^{ln}\bigg] \label{logdereps1}
\end{equation}
if $\varepsilon=1$ (the constant is the sum of $\frac{1}{2}$ and $\frac{\beta}{1-\beta}$, and the denominator does not vanish we exclude the case where $\delta=1$), and
\begin{equation}
2\pi i\bigg[\frac{\varepsilon}{2}+\sum_{n=1}^{\infty}\sum_{l=1}^{\infty}\Big(\beta^{l}q^{l[n-\frac{1+\varepsilon}{2}]}-\overline{\beta}^{l}q^{l[n-\frac{1-\varepsilon}{2}]}\Big)\bigg] \label{logdergen}
\end{equation}
for $0\leq\varepsilon<1$. Observe that Equation \eqref{logdereps1} with $\delta=0$ (and $\beta=-1$) also yields the vanishing of the theta derivative $\theta\big[{1 \atop 0}\big]'(\tau,0)$. As in Equation \eqref{logdergen} with $\varepsilon=0$ we get just a sum of the form $2\pi i\sum_{n=1}^{\infty}\sum_{l=1}^{\infty}(\beta^{l}-\overline{\beta}^{l})q^{l(n-1/2)}$, the vanishing of $\theta\big[{0 \atop 0}\big]'(\tau,0)$ and $\theta\big[{0 \atop 1}\big]'(\tau,0)$ is also easily seen (since $\beta=\pm1$ then). It remains to find out for which characteristics $\big[{\varepsilon \atop \delta}\big]$ (under our restrictions) the series from Equations \eqref{logdereps1} and \eqref{logdergen} yield a known function.

\smallskip

Our analysis will be based on the identification of the Fourier expansion of the following theta series, which are roughly the \emph{Hecke theta series} (with trivial character) associated to the imaginary quadratic fields $\mathbb{Q}(i=\sqrt{-1})$, $\mathbb{Q}(\sqrt{-2})$, and $\mathbb{Q}(\sqrt{-3})$ in the language of Chapter 5 of \cite{[K]} (and others). For $D\in\{1,2,3\}$ we denote the ring of integers in the field $\mathbb{Q}(\sqrt{-3})$ by $\mathcal{O}_{D}$, and set $\Theta_{D}(\tau)$ to be $\sum_{\alpha\in\mathcal{O}_{D}}q^{N(\alpha)/2}$, where $N(\alpha)$ is the norm of $\alpha$. As the theta function of a positive definite (odd) rank 2 lattice, it is well-known to be modular of weight 1 (with some level), and we would also like to express it as a quadratic expression in theta constants with integral characteristics. As we shall later identify the series $\Theta_{D}$ below via its Fourier expansion, we shall require some properties of its coefficients, for which we write $\Theta_{D}(\tau)=1+\sum_{N=1}^{\infty}s_{D}(N)q^{N/2}$ for every such $D$.
\begin{prop}
For every positive integer $N$ we have $s_{1}(N)=4\sum_{2\not|d|N}\big(\tfrac{-1}{d}\big)$, $s_{2}(N)=2\sum_{2\not|d|N}\big(\tfrac{-2}{d}\big)$, and $s_{3}(N)=6\sum_{d|N}\big(\tfrac{d}{3}\big)$. In addition, the following equalities hold:
\[\Theta_{1}(\tau)=\theta\big[{\textstyle{0 \atop 0}}\big]^{2}(\tau,0),\qquad\Theta_{2}(\tau)=\theta\big[{\textstyle{0 \atop 0}}\big](\tau,0)\theta\big[{\textstyle{0 \atop 0}}\big](2\tau,0),\] and
\[\Theta_{3}(\tau)=\theta\big[{\textstyle{0 \atop 0}}\big](\tau,0)\theta\big[{\textstyle{0 \atop 0}}\big](3\tau,0)+\theta\big[{\textstyle{1 \atop 0}}\big](\tau,0)\theta\big[{\textstyle{1 \atop 0}}\big](3\tau,0).\] \label{theta1}
\end{prop}
The expressions for the coefficients $s_{D}(N)$ follow from the fact that the imaginary quadratic fields $\mathbb{Q}(i)$, $\mathbb{Q}(\sqrt{-2})$, and $\mathbb{Q}(\sqrt{-3})$ have class number 1 and unit groups of orders 4, 2, and 6 respectively, so that after dividing by the number of units they become multiplicative and are easily verified for prime powers. The first two theta equalities are stated in \cite{[M1]}, with references to \cite{[B2]} and \cite{[D]} for proofs. For part $(iii)$ we also observe that the norm $a^{2}-ab+b^{2}$ of $a+b\zeta_{3}$ can be written as $\big(a-\frac{b}{2}\big)^{2}+3\big(\frac{b}{2}\big)^{2}$, so that the sum with even $b=2c$ yields the first summand, while the sum with odd $b=2c+1$ produces the second one.

Alternative expressions for the coefficients $s_{D}(N)$ that will be useful for recognizing the Fourier expansion of the functions $\Theta_{D}$ are as follows. We shall denote, roughly following \cite{[M1]}, the number of divisors $d$ of some number $N$, that are congruent to some number $r$ modulo a residue $t$, by $\sigma_{r+t\mathbb{Z}}(N)$.
\begin{lem}
The following equalities hold for any integer $N\geq1$:
\[(i)\qquad s_{1}(N)=4\sigma_{1+4\mathbb{Z}}(N)-4\sigma_{3+4\mathbb{Z}}(N)=-4i\sum_{2\not|d|N}i^{d}=-2i\sum_{d|N}\big(i^{d}-(-i)^{d}\big).\] Moreover, we have $s_{1}(N)=s_{1}(2N)$ for every such $N$.
\[(ii)\qquad s_{2}(N)=2\sigma_{1+8\mathbb{Z}}(N)+2\sigma_{3+8\mathbb{Z}}(N)-2\sigma_{5+8\mathbb{Z}}(N)-2\sigma_{7+8\mathbb{Z}}(N)=\] \[-\big(\tfrac{-2}{t}\big)\sqrt{2}i\sum_{2\not|d|N}(\zeta_{8}^{td}-\overline{\zeta}_{8}^{td})=\big(\tfrac{2}{t}\big)\sqrt{2}\sum_{2\not|d|N}\big(\tfrac{-1}{d}\big)(\zeta_{8}^{td}+\overline{\zeta}_{8}^{td})\] for every odd $t$. Also here we have $s_{2}(N)=s_{2}(2N)$ for every $N$.
\[(iii)\qquad2s_{2}(N)+s_{1}(N)=8\sigma_{1+8\mathbb{Z}}(N)-8\sigma_{7+8\mathbb{Z}}(N)\] and \[2s_{2}(N)-s_{1}(N)=8\sigma_{3+8\mathbb{Z}}(N)-8\sigma_{5+8\mathbb{Z}}(N).\]
\[(iv)\qquad s_{3}(N)=6\sigma_{1+3\mathbb{Z}}(N)-6\sigma_{2+3\mathbb{Z}}(N)=-2\sqrt{3}i\sum_{d|N}(\zeta_{3}^{d}-\overline{\zeta}_{3}^{d}).\] The invariance here takes the form $s_{3}(N)=s_{3}(3N)$ for every $N$. \label{recogser}
\end{lem}
The proof is very simple and straightforward. One just evaluates the characters and the powers of the roots of unity involved. The invariance under multiplying $N$ by powers of the appropriate prime is also immediate from the fact that the set of divisors effectively contributing to the coefficient in question remains unchanged in every case (for $s_{1}$ and $s_{2}$ it is related to some of the assertions in Theorems 7.3 and 7.5 of \cite{[M3]}).

Recall that decomposing the sums defining $\theta\big[{0 \atop 0}\big](\tau,0)$ and $\theta\big[{0 \atop 1}\big](\tau,0)$ in Equation \eqref{defser} according to the parity of the summation index $n$ yields the relations
\begin{equation}
\theta\big[{\textstyle{0 \atop 0}}\big](\tau,0)=\theta\big[{\textstyle{0 \atop 0}}\big](4\tau,0)+\theta\big[{\textstyle{1 \atop 0}}\big](4\tau,0)=2\theta\big[{\textstyle{0 \atop 0}}\big](4\tau,0)-\theta\big[{\textstyle{0 \atop 1}}\big](\tau,0), \label{pardecom}
\end{equation}
from which we also obtain
\[\theta\big[{\textstyle{0 \atop 0}}\big](\tau,0)=2\theta\big[{\textstyle{1 \atop 0}}\big](4\tau,0)+\theta\big[{\textstyle{0 \atop 1}}\big](\tau,0)\mathrm{\ and\ }\theta\big[{\textstyle{0 \atop 1}}\big](\tau,0)=\theta\big[{\textstyle{0 \atop 0}}\big](4\tau,0)-\theta\big[{\textstyle{1 \atop 0}}\big](4\tau,0).\] We remark that Equation \eqref{pardecom} combines with Proposition \ref{intchar} to yield the eta identities
\begin{equation}
\frac{\eta^{5}(\tau)}{\eta^{2}\big(\tfrac{\tau}{2}\big)\eta^{2}(2\tau)}\!=\frac{\eta^{6}(4\tau)+2\eta^{2}(2\tau)\eta^{4}(8\tau)}{\eta^{2}(2\tau)\eta(4\tau)\eta^{2}(8\tau)}=
\frac{2\eta(\tau)\eta^{5}(4\tau)-\eta^{2}\big(\tfrac{\tau}{2}\big)\eta^{2}(2\tau)\eta^{2}(8\tau)}{\eta(\tau)\eta^{2}(2\tau)\eta^{2}(8\tau)}= \label{etaid}
\end{equation}
\[=\frac{4\eta(\tau)\eta^{2}(8\tau)+\eta^{2}\big(\tfrac{\tau}{2}\big)\eta(4\tau)}{\eta(\tau)\eta(4\tau)},\mathrm{\ as\ well\ as\ }\frac{\eta^{2}\big(\frac{\tau}{2}\big)}{\eta(\tau)}=\frac{\eta^{6}(4\tau)-2\eta^{2}(2\tau)\eta^{4}(8\tau)}{\eta^{2}(2\tau)\eta(4\tau)\eta^{2}(8\tau)},\]
and combining these equalities yields additional eta identities.

We now prove some useful equivalents of Equation \eqref{pardecom} for the series from Proposition \ref{theta1} and the coefficients from Lemma \ref{recogser}.
\begin{cor}
The following Fourier expansions define explicit combinations of theta functions:
\[(i)\ \ \ \sum_{2\not|N}s_{1}(N)q^{N/2}=\Theta_{1}(\tau)-\Theta_{1}(2\tau)=\theta\big[{\textstyle{0 \atop 0}}\big]^{2}(\tau,0)-\theta\big[{\textstyle{0 \atop 0}}\big]^{2}(2\tau,0)=\theta\big[{\textstyle{1 \atop 0}}\big]^{2}(2\tau,0)\] and \[1+\sum_{N=1}^{\infty}(-1)^{N}s_{1}(N)q^{N/2}=2\Theta_{1}(2\tau)-\Theta_{1}(\tau)=\theta\big[{\textstyle{0 \atop 0}}\big]^{2}(2\tau,0)-\theta\big[{\textstyle{1 \atop 0}}\big]^{2}(2\tau,0).\]
\[(ii)\qquad\sum_{2\not|N}s_{2}(N)q^{N/2}=\Theta_{2}(\tau)-\Theta_{2}(2\tau)=\theta\big[{\textstyle{0 \atop 0}}\big](2\tau,0)\theta\big[{\textstyle{1 \atop 0}}\big](4\tau,0)\] and \[1+\sum_{N=1}^{\infty}(-1)^{N}s_{2}(N)q^{N/2}=2\Theta_{2}(2\tau)-\Theta_{2}(\tau)=\theta\big[{\textstyle{0 \atop 0}}\big](2\tau,0)\theta\big[{\textstyle{0 \atop 1}}\big](\tau,0).\]
\[(iii)\qquad\sum_{2\not|N}\big(\tfrac{-1}{N}\big)s_{1}(N)q^{N/2}=\Theta_{1}(\tau)-\Theta_{1}(2\tau)=\theta\big[{\textstyle{1 \atop 0}}\big]^{2}(2\tau,0)\] again, and
\[\sum_{2\not|N}\big(\tfrac{-1}{N}\big)s_{2}(N)q^{N/2}=\theta\big[{\textstyle{1 \atop 0}}\big](4\tau,0)\theta\big[{\textstyle{0 \atop 1}}\big](2\tau,0).\] \label{pow2twist}
\end{cor}

\begin{proof}
The leftmost equalities in all the equations in parts $(i)$ and $(ii)$ follow from the invariance of the coefficients $s_{1}(N)$ and $s_{2}(N)$ under multiplying $N$ by a power of 2 (see parts $(i)$ and $(ii)$ of Lemma \ref{recogser}). We then substitute the formulae from parts $(i)$ and $(ii)$ of Proposition \ref{theta1}, use Lemma 1.6 of Chapter 2 of \cite{[FK]} (with trivial characteristics) for the relations with $\theta\big[{1 \atop 0}\big]^{2}(2\tau,0)$ in part $(i)$, and the final expressions in part $(ii)$ are consequences of Equation \eqref{pardecom}. The first assertion in part $(iii)$ is immediate since it is known that the power series from part $(i)$ of Proposition \ref{theta1} does not contain any power of $q^{1/2}$ whose exponent is congruent to 3 modulo 4. For the remaining assertion we consider the residues modulo 4 of the powers appearing in the series $\theta\big[{0 \atop 0}\big](2\tau,0)\theta\big[{1 \atop 0}\big](4\tau,0)$ from part $(ii)$ here. Since odd squares are congruent to 1 modulo 4 (and even squares are divisible by 4), when we decompose $\theta\big[{\textstyle{0 \atop 0}}\big](2\tau,0)$ as in Equation \eqref{pardecom} we see that the product $\theta\big[{0 \atop 0}\big](8\tau,0)\theta\big[{1 \atop 0}\big](4\tau,0)$ yields only powers of $q^{N/2}$ with $N\equiv1(\mathrm{mod\ }4)$, while $\theta\big[{1 \atop 0}\big](8\tau,0)\theta\big[{1 \atop 0}\big](4\tau,0)$ involves only exponents that are congruent to 3 modulo 4 (in fact, such $N$ must lie in $1+8\mathbb{Z}$ or $3+8\mathbb{Z}$ respectively, since $s_{2}(N)=0$ when $N$ is congruent to 5 or 7 modulo 8). Hence inverting the sign of the coefficients that give residue 3 modulo 4 just corresponds to inverting the sign of
$\theta\big[{1 \atop 0}\big](8\tau,0)\theta\big[{1 \atop 0}\big](4\tau,0)$, and another application of Equation \eqref{pardecom} shows that this operation produces the asserted expression. This completes the proof of the corollary.
\end{proof}
We note that Corollary \ref{pow2twist} (as well as Equation \eqref{pardecom} before it) can be seen as some variant of the sign transform from Section 1.2 of \cite{[K]}. Therefore the expressions with $\Theta_{1}$ and $\Theta_{2}$ there can also be expressed by translations of the variable $\tau$ (including the latter assertion of part $(iii)$ there). However, we prefer to keep just multiples of $\tau$ as the variable throughout, and there does not seem to be a simple expression using $\Theta_{2}$ for the last series in Corollary \ref{pow2twist}. We shall therefore use only expressions in terms of theta functions with characteristics, of weight $\frac{1}{2}$, in Section \ref{Charin1/4Z} below.

We shall also be needing some properties and variants of the series $\Theta_{3}$.
\begin{cor}
Twisting the Fourier expansion of $\Theta_{3}$ yields the following functions:
\[(i)\qquad\sum_{2\not|N}s_{3}(N)q^{N/2}=\Theta_{3}(\tau)-\Theta_{3}(4\tau),\quad\sum_{3\not|N}s_{3}(N)q^{N/2}=\Theta_{3}(\tau)-\Theta_{3}(3\tau),\] and \[1+\sum_{N=1}^{\infty}(-1)^{N}s_{3}(N)q^{N/2}=2\Theta_{3}(4\tau)-\Theta_{3}(\tau).\]
$(ii)$ Modify $s_{3}(N)$ to set $s_{3}^{o}(N)=6\sum_{2\not|d|N}\big(\tfrac{d}{3}\big)=-2\sqrt{3}i\sum_{2\not|d|N}(\zeta_{3}^{d}-\overline{\zeta}_{3}^{d})$ as well as $s_{3}^{\hat{o}}(N)=6\sum_{d|N,\ 2\not|N/d}\big(\tfrac{d}{3}\big)=-2\sqrt{3}i\sum_{d|N,\ 2\not|N/d}(\zeta_{3}^{d}-\overline{\zeta}_{3}^{d})$. Then
\[2+\sum_{N=1}^{\infty}s_{3}^{o}(N)q^{N/2}=\Theta_{3}(\tau)+\Theta_{3}(2\tau)\quad\mathrm{and}\quad\sum_{N=1}^{\infty}s_{3}^{\hat{o}}(N)q^{N/2}=\Theta_{3}(\tau)-\Theta_{3}(2\tau).\]
$(iii)$ Similarly let $s_{3}^{s}(N)=6\sum_{d|N}(-1)^{d}\big(\tfrac{d}{3}\big)$ and $s_{3}^{\hat{s}}(N)=6\sum_{d|N}(-1)^{N/d}\big(\tfrac{d}{3}\big)$. We then have
\[-3+\sum_{N=1}^{\infty}s_{3}^{s}(N)q^{N/2}=-\Theta_{3}(\tau)-2\Theta_{3}(2\tau)\mathrm{\ and\ }s_{3}^{s}(N)=+2\sqrt{3}i\sum_{d|N}(\zeta_{6}^{d}-\overline{\zeta}_{6}^{d}),\] as well as \[1+\sum_{N=1}^{\infty}s_{3}^{\hat{s}}(N)q^{N/2}=2\Theta_{3}(2\tau)-\Theta_{3}(\tau).\]
$(iv)$ The sum $-2\sqrt{3}i\sum_{2\not|d|N}(\zeta_{6}^{d}-\overline{\zeta}_{6}^{d})$ gives $s_{3}^{o}(N)$ back again, while the expression $6\sum_{2\not|d|N}(-1)^{N/d}\big(\frac{d}{3}\big)$ equals $(-1)^{N}s_{3}^{o}(N)$, with the series
\[2+\sum_{N=1}^{\infty}(-1)^{N}s_{3}^{o}(N)q^{N/2}=2\Theta_{3}(4\tau)-\Theta_{3}(\tau)+\Theta_{3}(2\tau).\]
On the other hand, the sum $-2\sqrt{3}i\sum_{d|N,\ 2\not|N/d}(\zeta_{6}^{d}-\overline{\zeta}_{6}^{d})$ yields $-(-1)^{N}s_{3}^{\hat{o}}(N)$, and we have \[-\sum_{N=1}^{\infty}(-1)^{N}s_{3}^{\hat{o}}(N)q^{N/2}=\Theta_{3}(\tau)+\Theta_{3}(2\tau)-2\Theta_{3}(4\tau).\]
$(v)$ Considering only indices that are not divisible by 3 produces the series
\[\sum_{3\not|N}s_{3}^{o}(N)q^{N/2}=\Theta_{3}(\tau)+\Theta_{3}(2\tau)-\Theta_{3}(3\tau)-\Theta_{3}(6\tau),\]
\[\sum_{3\not|N}s_{3}^{\hat{o}}(N)q^{N/2}=\Theta_{3}(\tau)-\Theta_{3}(2\tau)-\Theta_{3}(3\tau)+\Theta_{3}(6\tau),\]
\[\sum_{3\not|N}s_{3}^{s}(N)q^{N/2}=\Theta_{3}(3\tau)+2\Theta_{3}(6\tau)-\Theta_{3}(\tau)-2\Theta_{3}(2\tau),\]
\[\sum_{3\not|N}s_{3}^{\hat{s}}(N)q^{N/2}=2\Theta_{3}(2\tau)-\Theta_{3}(\tau)-2\Theta_{3}(6\tau)+\Theta_{3}(3\tau),\]
\[\sum_{3\not|N}(-1)^{N}s_{3}^{o}(N)q^{N/2}=2\Theta_{3}(4\tau)-\Theta_{3}(\tau)+\Theta_{3}(2\tau)-2\Theta_{3}(12\tau)+\Theta_{3}(3\tau)-\Theta_{3}(6\tau),\]
and
\[\sum_{3\not|N}(-1)^{N}s_{3}^{\hat{o}}(N)q^{N/2}=\Theta_{3}(3\tau)+\Theta_{3}(6\tau)-2\Theta_{3}(12\tau)-\Theta_{3}(\tau)-\Theta_{3}(2\tau)+2\Theta_{3}(4\tau).\] \label{pow3twist}
\end{cor}

\begin{proof}
The fact that multiplying a divisor $d$ of $N$ by 2 inverts the sign $\big(\frac{d}{3}\big)$ or the value of $\zeta_{3}^{d}-\overline{\zeta}_{3}^{d}$ implies that if $2d|N$ then the contributions of $d$ and $2d$ to $s_{3}(N)$ cancel. Hence if the power of 2 dividing $N$ is even then one may consider the coefficient as arising only from the contributions arising either from odd $d$ or from $d$ with odd $\frac{N}{d}$, while if this power of 2 is odd then $s_{3}(N)=0$. The first assertion in part $(i)$ is then established, and the third assertion there follows as in the proof of parts $(i)$ and $(ii)$ of Corollary \ref{pow2twist}. The second assertion of that part is immediate from the invariance of these coefficients under multiplication of $N$ by 3 (by part $(iv)$ of Lemma \ref{recogser}). The analysis of the powers of 2 also implies that both $s_{3}^{o}(N)$ and $s_{3}^{\hat{o}}(N)$ equal $s_{3}(N)$ when 2 divides $N$ to an even power, while if that power is odd then $s_{3}^{o}(N)$ is $s_{3}\big(\frac{N}{2}\big)$ but $s_{3}^{\hat{o}}(N)$ equals $-s_{3}\big(\frac{N}{2}\big)$ (the equivalence of the two expressions that we used for defining both $s_{3}^{o}(N)$ and $s_{3}^{\hat{o}}(N)$ is proved precisely like the corresponding expression for $s_{3}(N)$ in $(iv)$ of Lemma \ref{recogser}). This yields part $(ii)$ as well, by adding the appropriate constant coefficients to the series. Since $s_{3}^{s}(N)$ (resp. $s_{3}^{\hat{s}}(N)$) clearly equals $s_{3}(N)-2s_{3}^{o}(N)$ (resp. $s_{3}(N)-2s_{3}^{\hat{o}}(N)$) for every $N$, part $(ii)$ implies the two formulae with the series in part $(iii)$. As $\zeta_{6}=-\overline{\zeta}_{3}$ and $\overline{\zeta}_{6}=-\zeta_{3}$, the difference $\zeta_{6}^{d}-\overline{\zeta}_{6}^{d}$ equals $-(-1)^{d}(\zeta_{3}^{d}-\overline{\zeta}_{3}^{d})$ (as is easily verified for even and odd $d$ separately), and the remaining formula in part $(iii)$ follows, as well as the first assertion in part $(iv)$. For the other sum involving $\zeta_{6}$, the sign $(-1)^{d}$ coincides with $(-1)^{N}$ when $\frac{N}{d}$ is odd, and the resulting values follow from part $(ii)$. Writing now $s_{3}^{o}(N)$ (resp. $s_{3}^{\hat{o}}(N)$) as $s_{3}^{\hat{o}}(N)$ for every $N$ minus $2s_{3}(N)$ when $N$ is odd (where both $s_{3}^{o}(N)$ and $s_{3}^{\hat{o}}(N)$ are just $s_{3}(N)$), the relevant formulae from parts $(i)$ and $(ii)$ combine (with the appropriate signs) to the asserted functions from part $(iv)$. Since the invariance of $s_{3}(N)$ under the operation of multiplying $N$ by 3 (see part $(iv)$ of Lemma \ref{recogser}, which clearly extends to the numbers $s_{3}^{o}(N)$, $s_{3}^{\hat{o}}(N)$, $s_{3}^{s}(N)$ and $s_{3}^{\hat{s}}(N)$), the formulae in part $(v)$ follow directly from subtracting the expressions from parts $(ii)$ or $(iii)$ with $\tau$ and with $3\tau$ from one another. This proves the corollary.
\end{proof}

We remark that a variant of the duplication formula from Lemma 1.6 of Chapter 2 of \cite{[FK]} exists for $\Theta_{3}$, but a product of $\theta\big[{\varepsilon \atop \delta}\big](\tau,0)$ with $\theta\big[{\lambda \atop \mu}\big](3\tau,0)$ would involve four terms (rather than two), and the variables in the resulting multipliers would be $4\tau$ and $12\tau$. Using Equation \eqref{pardecom} and its variants, as well as the theta identities from Corollary \ref{Foureta1/2} below, one could show that the combinations $\Theta_{3}(\tau)+2\Theta_{3}(4\tau)$ and $4\Theta_{3}(4\tau)-\Theta_{3}(\tau)$ reduce to $3\theta\big[{0 \atop 0}\big](\tau,0)\theta\big[{0 \atop 0}\big](3\tau,0)$ and $3\theta\big[{0 \atop 1}\big](\tau,0)\theta\big[{0 \atop 1}\big](3\tau,0)$ respectively. However, since this simplifies only very few expressions in Corollary \ref{pow3twist}, we shall leave the expressions from that corollary (and its consequences in Section \ref{Charin1/3Z} below) as they are in what follows.

\section{Theta Derivatives with Characteristics from $\frac{1}{4}\mathbb{Z}$ \label{Charin1/4Z}}

Theorem 1 of \cite{[M1]} identifies the theta derivative $\theta\big[{1 \atop 1/2}\big]'(\tau,0)$ as the expression $-\pi\theta\big[{1 \atop 1/2}\big](\tau,0)\theta\big[{0 \atop 0}\big]^{2}(2\tau,0)$. Indeed, when one substitutes the values $\varepsilon=1$ and $\delta=\frac{1}{2}$ (hence $\beta=-\mathbf{e}\big(-\frac{\delta}{2}\big)=i$) into the expression from Equation \eqref{logdereps1}, and writes $N=nl$ and $d=l$, the terms with even $d$ cancel out, the ones with odd $d$ give $+4\pi i \cdot i^{d}$, and the constant coefficient $2\pi i\frac{1+i}{2(1-i)}$ equals $-\pi$. The assertion thus follows from part $(i)$ of Lemma \ref{recogser} (but note that the power series is with $q^{N}$ rather than $q^{N/2}$). Let us now show how these formulae can be used for establishing expressions for some other theta derivatives.
\begin{thm}
Explicit expressions for the theta derivatives $\theta\big[{0 \atop 1/2}\big]'(\tau,0)$ and $\theta\big[{1/2 \atop 1}\big]'(\tau,0)$ are given by
\begin{equation}
-\pi\theta\big[{\textstyle{0 \atop 1/2}}\big](\tau,0)\Big(\theta\big[{\textstyle{0 \atop 0}}\big]^{2}(\tau,0)-\theta\big[{\textstyle{0 \atop 0}}\big]^{2}(2\tau,0)\Big)=-\pi\theta\big[{\textstyle{0 \atop 1/2}}\big](\tau,0)\theta\big[{\textstyle{1 \atop 0}}\big]^{2}(2\tau,0), \label{01/2}
\end{equation}
and
\begin{equation}
\tfrac{\pi i}{2}\theta\big[{\textstyle{1/2 \atop 1}}\big](\tau,0)\theta\big[{\textstyle{0 \atop 0}}\big]^{2}\big(\tfrac{\tau}{2},0\big)=
\tfrac{\pi i}{2}\theta\big[{\textstyle{1/2 \atop 1}}\big](\tau,0)\Big(\theta\big[{\textstyle{0 \atop 0}}\big]^{2}(\tau,0)+\theta\big[{\textstyle{1 \atop 0}}\big]^{2}(\tau,0)\Big) \label{1/21}
\end{equation}
respectively, and the equality
\begin{equation}
\theta\big[{\textstyle{1/2 \atop 0}}\big]'(\tau,0)=\tfrac{\pi i}{2}\theta\big[{\textstyle{1/2 \atop 0}}\big](\tau,0)\Big(\theta\big[{\textstyle{0 \atop 0}}\big]^{2}(\tau,0)-\theta\big[{\textstyle{1 \atop 0}}\big]^{2}(\tau,0)\Big) \label{1/20}
\end{equation}
also holds. In addition, we have the equalities
\begin{equation}
\theta\big[{\textstyle{1/2 \atop 1/2}}\big]'(\tau,0)=\tfrac{\pi i}{2}\theta\big[{\textstyle{1/2 \atop 1/2}}\big](\tau,0)\Big(\theta\big[{\textstyle{0 \atop 0}}\big]^{2}(2\tau,0)-\theta\big[{\textstyle{1 \atop 0}}\big]^{2}(2\tau,0)+i\theta\big[{\textstyle{1 \atop 0}}\big]^{2}(\tau,0)\Big) \label{1/21/2}
\end{equation}
and
\begin{equation}
\theta\big[{\textstyle{1/2 \atop 3/2}}\big]'(\tau,0)=\tfrac{\pi i}{2}\theta\big[{\textstyle{1/2 \atop 3/2}}\big](\tau,0)\Big(\theta\big[{\textstyle{0 \atop 0}}\big]^{2}(2\tau,0)-\theta\big[{\textstyle{1 \atop 0}}\big]^{2}(2\tau,0)-i\theta\big[{\textstyle{1 \atop 0}}\big]^{2}(\tau,0\big)\Big). \label{1/23/2}
\end{equation} \label{2chinZ}
\end{thm}

\begin{proof}
With $\varepsilon=0$ and $\delta=\frac{1}{2}$ (hence with $\beta=i$ again), Equation \eqref{logdergen} takes the form $2\pi i\sum_{n=1}^{\infty}\sum_{l=1}^{\infty}\big(i^{l}-(-i)^{l}\big)q^{l(n-1/2)}$. Again the terms with even $l$ vanish, so that only powers $q^{N/2}$ with odd $N$ survive. As the coefficient of $q^{N/2}$ is $-\pi s_{1}(N)$ by part $(i)$ of Lemma \ref{recogser}, Equation \eqref{01/2} follows from part $(i)$ of Corollary \ref{pow2twist}. For Equation \eqref{1/21} we set $\varepsilon=\frac{1}{2}$ and $\delta=1$ (so that $\beta=1$) in Equation \eqref{logdergen}, and get a constant coefficient of $\frac{\pi i}{2}$, terms of the form $q^{l(4n-3)/4}$ with the coefficient $2\pi i$, and terms $q^{l(4n-1)/4}$ multiplied by $-2\pi i$. Writing the exponent of $q$ as $\frac{N}{4}$ and taking the divisor $d$ to be $4n-3$ or $4n-1$ (positive, with residues 1 and 3 modulo 4 respectively), the resulting coefficient of $q^{N/4}$ (for every $N$) is $\frac{\pi i}{2}s_{1}(N)$ by another application of part $(i)$ of Lemma \ref{recogser}, proving the desired result via part $(i)$ of Proposition \ref{theta1}. Now replace $\delta$ from 1 to 0, hence $\beta$ from 1 to $-1$, in the latter argument, so that each term involving an index $l$ is multiplied by $(-1)^{l}$. Since $N$ is $l$ times an odd number, this sign is the same as $(-1)^{N}$, and Equation \eqref{1/20} thus follows from part $(i)$ of Corollary \ref{pow2twist}.

Turning to the case where $\varepsilon=\frac{1}{2}$ and $\delta$ is either $\frac{1}{2}$ or $\frac{3}{2}$, a parameter that we shall write as $\frac{2\mp1}{2}$, we find that $\beta=\pm i$. An argument similar to the previous paragraph yields the same constant coefficient, while now $q^{l(4n-3)/4}$ comes multiplied by $2\pi i(\pm i)^{l}$ and $q^{l(4n-1)/4}$ has a multiplier of $-2\pi i(\mp i)^{l}$. The terms with even $l$ (or equivalently even $N$) can be analyzed just in the proof of Equation \eqref{1/20} (but with the variable multiplied by 2---we are also implicitly using the last assertion of part $(i)$ of Lemma \ref{recogser}), thus yielding the common terms in Equations \eqref{1/21/2} and \eqref{1/23/2}. For odd $N$ we can take the sign out of the powers, and obtain a coefficient of $\pm2\pi i \cdot i^{l}$ in front of both terms $q^{l(4n-3)/4}$ and $q^{l(4n-1)/4}$, i.e., in front of any term of the form $q^{l(2m-1)/4}$ with non-negative $m$ (and odd $l$). The resulting coefficient of $q^{N/4}$ (for odd $N$) is again $\mp\frac{\pi}{2}s_{1}(N)=\frac{\pi i}{2}\cdot\pm is_{1}(N)$ by part $(i)$ of Lemma \ref{recogser} (with $d=l$), which yields the remaining terms in Equations \eqref{1/21/2} and \eqref{1/23/2} by part $(i)$ of Corollary \ref{pow2twist}. This completes the proof of the theorem.
\end{proof}
We remark that Equation \eqref{01/2} appears as Equation (3.1) of \cite{[M3]}.

We recall that \cite{[M1]} applies its Theorem 1 for obtaining the Fourier expansion
\begin{equation}
\frac{\eta^{9}(2\tau)}{\eta^{3}(\tau)\eta^{3}(4\tau)}=\sum_{n=1}^{\infty}(-1)^{n(n-1)/2}(2n+1)q^{(2n+1)^{2}/8}=\sum_{n=1}^{\infty}n\big(\tfrac{-2}{n}\big)q^{n^{2}/8} \label{eta3strans}
\end{equation}
(see Equation (8) of that reference). Equation \eqref{eta3strans} also appears (up to scalar multiplication of the variable) in Theorem 1.2 of \cite{[LO]} and in Equation (8.16) in Theorem 8.5 of \cite{[K]}. It can also be seen, in the language of Section 1.2 of \cite{[K]} and others, as the \emph{sign transform} of the series for $\eta^{3}(\tau)$ appearing in our Equation \eqref{eta3exp}. The idea is that arguments similar to the proof of Proposition \ref{intchar} combine with the fact that the product of $1-iq^{n}$ and $1+iq^{n}$ is $1+q^{2n}$ to show that $\theta\big[{1 \atop 1/2}\big](\tau,0)$ is the eta quotient $\sqrt{2}\frac{\eta(\tau)\eta(4\tau)}{\eta(2\tau)}$ (the leading coefficient is the product of $\zeta_{8}$ and $1-i$). This, together with Proposition \ref{intchar} itself and the derivative of Equation \eqref{defser} (with some summation index change), yields Equation \eqref{eta3strans}.

The argument proving the equality $\theta\big[{1 \atop 1/2}\big](\tau,0)=\sqrt{2}\frac{\eta(\tau)\eta(4\tau)}{\eta(2\tau)}$, together with the observation that the product $\prod_{n=1}^{\infty}(1 \pm q^{n-1/4})(1 \pm q^{n-3/4})$ reduces to $\prod_{n=1}^{\infty}(1 \pm q^{(2n-1)/4})$, combine to establish the following result, which is essentially contained, via the expansions of theta series, in Theorem 1.1 of \cite{[LO]} and Theorem 8.1 of \cite{[K]} (up to the usual rescalings).
\begin{prop}
The following equalities hold: \[\theta\big[{\textstyle{0 \atop 1/2}}\big](\tau,0)=\frac{\eta^{2}(2\tau)}{\eta(4\tau)},\ \theta\big[{\textstyle{1/2 \atop 1}}\big](\tau,0)=\zeta_{8}\frac{\eta(\tau)\eta\big(\frac{\tau}{4}\big)}{\eta\big(\frac{\tau}{2}\big)},\mathrm{\ and\ }\theta\big[{\textstyle{1/2 \atop 0}}\big](\tau,0)=\frac{\eta^{2}\big(\frac{\tau}{2}\big)}{\eta\big(\frac{\tau}{4}\big)}.\] \label{halfchar}
\end{prop}

From this we deduce some identities.
\begin{cor}
The theta constants with half-integral characteristics satisfy the identities $\theta\big[{1 \atop 1/2}\big](\tau,0)=(1-i)\theta\big[{1/2 \atop 1}\big](4\tau,0)$, $\theta\big[{0 \atop 1/2}\big](\tau,0)=\theta\big[{0 \atop 1}\big](4\tau,0)$, and $\theta\big[{1 \atop 0}\big](\tau,0)=2\theta\big[{1/2 \atop 0}\big](4\tau,0)$. We also have the eta identities
\[\frac{\eta^{14}(\tau)\eta^{4}(4\tau)-\eta^{14}(2\tau)\eta^{4}\big(\frac{\tau}{2}\big)}{\eta^{4}\big(\frac{\tau}{2}\big)\eta^{4}(\tau)\eta^{2}(2\tau)\eta^{5}(4\tau)}=4\eta^{3}(4\tau),\]
\[\eta\big(\tfrac{\tau}{4}\big)\frac{\eta^{12}(\tau)+4\eta^{8}(2\tau)\eta^{4}\big(\frac{\tau}{2}\big)}{\eta^{5}\big(\frac{\tau}{2}\big)\eta(\tau)\eta^{4}(2\tau)}=
\frac{\eta^{9}\big(\frac{\tau}{2}\big)}{\eta^{3}(\tau)\eta^{3}\big(\frac{\tau}{4}\big)},\]
and
\[\frac{\eta^{12}(\tau)-\eta^{8}(2\tau)\eta^{4}\big(\frac{\tau}{2}\big)}{\eta\big(\frac{\tau}{4}\big)\eta^{2}\big(\frac{\tau}{2}\big)\eta^{2}(\tau)\eta^{4}(2\tau)}=\eta^{3}\big(\tfrac{\tau}{4}).\] \label{Foureta1/2}
\end{cor}

\begin{proof}
The theta constant identities are immediate consequences of Propositions \ref{intchar} and \ref{halfchar} (recall that $\frac{\sqrt{2}}{\zeta_{8}}=1-i$). On the other hand, the derivative of Equation \eqref{defser} shows that the Fourier expansions of the theta derivatives $\theta\big[{0 \atop 1/2}\big]'(\tau,0)$, $\theta\big[{1/2 \atop 1}\big]'(\tau,0)$, and $\theta\big[{1/2 \atop 0}\big]'(\tau,0)$ are $-4\pi\sum_{n=0}^{\infty}n\big(\tfrac{-4}{n}\big)q^{n^{2}/2}$, $\frac{\pi i}{2}\zeta_{8}\sum_{n=1}^{\infty}n\big(\tfrac{-2}{n}\big)q^{n^{2}/32}$ and $\frac{\pi i}{2}\sum_{n=0}^{\infty}n\big(\tfrac{-4}{n}\big)q^{n^{2}/32}$ respectively (due to cancelations for even $n$ in the first series and a change of the summation index in the other two). But as these series are $-4\pi\eta^{3}(4\tau)$, $\frac{\pi i}{2}\zeta_{8}\frac{\eta^{9}(\tau/2)}{\eta^{3}(\tau)\eta^{3}(\tau/4)}$, and $\frac{\pi i}{2}\eta^{3}\big(\frac{\tau}{4})$ respectively by Equations \eqref{eta3exp} and \eqref{eta3strans}, we substitute these expressions into Equations \eqref{01/2}, \eqref{1/21}, and \eqref{1/20} respectively, divide by the relevant scalar ($-\pi$, $\frac{\pi i}{2}\zeta_{8}$, or $\frac{\pi i}{2}$), and express the right hand side of the appropriate equation using Propositions \ref{intchar} and \ref{halfchar}. This proves the corollary.
\end{proof}
The second expression in Equation \eqref{01/2}, with $\theta\big[{1 \atop 0}\big]^{2}(2\tau,0)$, indeed produces $4\eta^{3}(4\tau)$ again. Similarly, the expression involving $\theta\big[{0 \atop 0}\big]^{2}\big(\tfrac{\tau}{2},0\big)$ in Equation \eqref{1/21} yields indeed $\zeta_{8}\frac{\eta^{9}(\tau/2)}{\eta^{3}(\tau)\eta^{3}(\tau/4)}$. Unfortunately, the theta constants appearing in Equations \eqref{1/21/2} and \eqref{1/23/2} (as well as in all the formulae from Theorems \ref{4deltainZ} and \ref{4epsinZ}) cannot be expressed as eta quotients via such an argument (this is essentially proved in \cite{[LO]}). It is possible that the eta identities from Corollary \ref{Foureta1/2} can be deduced from Equation \eqref{etaid}, and that they are dependent of one another, but I have not checked these statements in detail.

\smallskip

Theorem 2 of \cite{[M1]} concerns the two theta derivatives $\theta\big[{1 \atop 1/4}\big]'(\tau,0)$ and $\theta\big[{1 \atop 3/4}\big]'(\tau,0)$, which we write as $\theta\big[{1 \atop (2\mp1)/4}\big]'(\tau,0)$ with the two possible signs. The two statements of that theorem can be written together as \[\theta\big[{\textstyle{1 \atop (2\mp1)/4}}\big]'(\tau,0)=-\pi\theta\big[{\textstyle{1 \atop (2\mp1)/4}}\big](\tau,0)\theta\big[{\textstyle{0 \atop 0}}\big](4\tau,0)\Big(\sqrt{2}\theta\big[{\textstyle{0 \atop 0}}\big](2\tau,0)\mp\theta\big[{\textstyle{0 \atop 0}}\big](4\tau,0)\Big).\] The proof goes along the same lines (just expressed a bit differently in \cite{[M1]}): With $\varepsilon=1$ and $\delta=\frac{2\mp1}{4}$ the parameter $\beta$ is $\zeta_{8}^{2\pm1}$, which we write as $\zeta_{8}^{h}$ for $h$ being 3 or 1 respectively. Then the equality $(1-\zeta_{8})(1-\zeta_{8}^{3})=-\sqrt{2}i$ implies that the constant $\pi i\frac{1+\zeta_{8}^{h}}{1-\zeta_{8}^{h}}$ from Equation \eqref{logdereps1} equals $-\pi(\sqrt{2}\pm1)$. For positive $N$ the coefficient multiplying $q^{N}$ in that equation is $2\pi i\sum_{l|N}(\zeta_{8}^{hl}-\overline{\zeta}_{8}^{hl})$ with our value of $\beta$, so that part $(ii)$ of Lemma \ref{recogser} implies that the odd divisors $l$ of $N$ contribute $-\sqrt{2}\pi s_{2}(N)$ (since both $h=3$ and $h=1$ satisfy $\big(\frac{-2}{h}\big)=+1$). For even $N$ we also have to consider the contributions of even divisors $l$, where the observation that even powers of $(\zeta_{8}^{h})^{2}$ are powers of $i^{h}=(-1)^{(h-1)/2}i$ combines with part $(i)$ of Lemma \ref{recogser} to show that these contributions produce $(-1)^{(h-1)/2}=\mp1$ times $-\pi s_{1}\big(\frac{N}{2}\big)$ (or equivalently $-\pi s_{1}(N)$). Expressing the resulting linear combination of $\Theta_{1}(2\tau)$ and $\Theta_{2}(2\tau)$ as in Proposition \ref{theta1} establishes the desired result. Note that these expressions can be compared with Theorem 4.2 of \cite{[M2]}, resulting in some interesting theta constant identities. On the other hand, several other theta derivatives can be determined in the manner just described.
\begin{thm}
Two theta derivatives with $\varepsilon=0$ are expressed in
\begin{equation}
\theta\big[{\textstyle{0 \atop 1/4}}\big]'(\tau,0)=-\pi\theta\big[{\textstyle{0 \atop 1/4}}\big](\tau,0)\Big(\sqrt{2}\theta\big[{\textstyle{0 \atop 0}}\big](2\tau,0)\theta\big[{\textstyle{1 \atop 0}}\big](4\tau,0)-\theta\big[{\textstyle{1 \atop 0}}\big]^{2}(4\tau,0)\Big) \label{01/4}
\end{equation}
and
\begin{equation}
\theta\big[{\textstyle{0 \atop 3/4}}\big]'(\tau,0)=-\pi\theta\big[{\textstyle{0 \atop 3/4}}\big](\tau,0)\Big(\sqrt{2}\theta\big[{\textstyle{0 \atop 0}}\big](2\tau,0)\theta\big[{\textstyle{1 \atop 0}}\big](4\tau,0)+\theta\big[{\textstyle{1 \atop 0}}\big]^{2}(4\tau,0)\Big). \label{03/4}
\end{equation}
In addition, four other theta derivatives are given by the formulae
\[\theta\big[{\textstyle{1/2 \atop 1/4}}\big]'(\tau,0)=\tfrac{\pi i}{2}\theta\big[{\textstyle{1/2 \atop 1/4}}\big](\tau,0)\Big[\theta\big[{\textstyle{0 \atop 0}}\big]^{2}(4\tau,0)-\theta\big[{\textstyle{1 \atop 0}}\big]^{2}(4\tau,0)+\]
\begin{equation}
-i\theta\big[{\textstyle{1 \atop 0}}\big]^{2}(2\tau,0)-\sqrt{2}\theta\big[{\textstyle{1 \atop 0}}\big](2\tau,0)\Big(\theta\big[{\textstyle{0 \atop 1}}\big](\tau,0)-i\theta\big[{\textstyle{0 \atop 0}}\big](\tau,0)\Big)\Big], \label{1/21/4}
\end{equation}
\[\theta\big[{\textstyle{1/2 \atop 3/4}}\big]'(\tau,0)=\tfrac{\pi i}{2}\theta\big[{\textstyle{1/2 \atop 3/4}}\big](\tau,0)\Big[\theta\big[{\textstyle{0 \atop 0}}\big]^{2}(4\tau,0)-\theta\big[{\textstyle{1 \atop 0}}\big]^{2}(4\tau,0)+\]
\begin{equation}
+i\theta\big[{\textstyle{1 \atop 0}}\big]^{2}(2\tau,0)+\sqrt{2}\theta\big[{\textstyle{1 \atop 0}}\big](2\tau,0)\Big(\theta\big[{\textstyle{0 \atop 1}}\big](\tau,0)+i\theta\big[{\textstyle{0 \atop 0}}\big](\tau,0)\Big)\Big], \label{1/23/4}
\end{equation}
\[\theta\big[{\textstyle{1/2 \atop 5/4}}\big]'(\tau,0)=\tfrac{\pi i}{2}\theta\big[{\textstyle{1/2 \atop 5/4}}\big](\tau,0)\Big[\theta\big[{\textstyle{0 \atop 0}}\big]^{2}(4\tau,0)-\theta\big[{\textstyle{1 \atop 0}}\big]^{2}(4\tau,0)+\]
\begin{equation}
-i\theta\big[{\textstyle{1 \atop 0}}\big]^{2}(2\tau,0)+\sqrt{2}\theta\big[{\textstyle{1 \atop 0}}\big](2\tau,0)\Big(\theta\big[{\textstyle{0 \atop 1}}\big](\tau,0)-i\theta\big[{\textstyle{0 \atop 0}}\big](\tau,0)\Big)\Big], \label{1/25/4}
\end{equation}
and \[\theta\big[{\textstyle{1/2 \atop 7/4}}\big]'(\tau,0)=\tfrac{\pi i}{2}\theta\big[{\textstyle{1/2 \atop 7/4}}\big](\tau,0)\Big[\theta\big[{\textstyle{0 \atop 0}}\big]^{2}(4\tau,0)-\theta\big[{\textstyle{1 \atop 0}}\big]^{2}(4\tau,0)+\]
\begin{equation}
+i\theta\big[{\textstyle{1 \atop 0}}\big]^{2}(2\tau,0)-\sqrt{2}\theta\big[{\textstyle{1 \atop 0}}\big](2\tau,0)\Big(\theta\big[{\textstyle{0 \atop 1}}\big](\tau,0)+i\theta\big[{\textstyle{0 \atop 0}}\big](\tau,0)\Big)\Big]. \label{1/27/4}
\end{equation} \label{4deltainZ}
\end{thm}

\begin{proof}
We write the first two characteristics as $\big[{0 \atop (2\mp1)/4}\big]$, and substitute into Equation \eqref{logdergen} the values $\varepsilon=0$ and $\delta=\frac{2\mp1}{4}$ (whence $\beta=\zeta_{8}^{2\pm1}=\zeta_{8}^{h}$ as in our description of the proof of Theorem 2 of \cite{[M1]}). This yields no constant term, and if we write $q^{l(2n-1)/2}$ as $q^{N/2}$ then only divisors $l$ of $N$ with odd $\frac{N}{l}$ have to be considered. Therefore if we write $N$ as $2^{a}N_{2}$ with odd $N_{2}$ then $l=2^{a}d$ for $d|N_{2}$ and the coefficient of $q^{N/2}$ is $2\pi i\sum_{d|N_{2}}(\zeta_{8}^{2^{a}dh}-\overline{\zeta}_{8}^{2^{a}dh})$. For odd $N$ the proof of Theorem 2 of \cite{[M1]} given above determines this coefficient as $-\sqrt{2}\pi s_{2}(N)$ for both values of $h$ (i.e., with both signs), and since these coefficients appear only for odd $N$, part $(ii)$ of Corollary \ref{pow2twist} implies that this part of the expansion gives the first term in Equations \eqref{01/4} and \eqref{03/4}. The coefficient in front of $q^{N/2}$ vanishes if $4|N$, while if $N\equiv2(\mathrm{mod\ }4)$ then it equals $\mp\pi s_{1}\big(\frac{N}{2}\big)$, or equivalently $\mp\pi s_{1}(N)$, by part $(i)$ of Lemma \ref{recogser}. But since only odd powers $\frac{N}{2}$ appear, part $(i)$ of Corollary \ref{pow2twist} yields the last terms in Equations \eqref{01/4} and \eqref{03/4}.

We now turn to the remaining 4 characteristics, which we write as $\big[{1/2 \atop (12-t)/4}\big]$ for $t$ being 5, and 7, 9, or 11. Then $\beta$ is just $\zeta_{8}^{t}$, and together with $\varepsilon=\frac{1}{2}$ the expression appearing in Equation \eqref{logdergen} yields the series with the constant term $\frac{\pi i}{2}$, and with terms of the form $2\pi i\zeta_{8}^{tl}q^{l(4n-3)/4}$ as well as $-2\pi i\overline{\zeta}_{8}^{tl}q^{l(4n-1)/4}$. Once again we look for the coefficient of $q^{N/4}$, and the parity of $N$ coincides with that of $l$. We first consider the terms associated with even indices $l$ (hence also $N$), where by writing $l=2m$ and $\zeta_{8}^{2t}=\pm i$ one easily sees that the proof of Equations \eqref{1/21/2} and \eqref{1/23/2} in Theorem \ref{2chinZ} transforms this part to the terms appearing in the first lines of Equations \eqref{1/21/4}, \eqref{1/23/4}, \eqref{1/25/4}, and \eqref{1/27/4}. Indeed, the description of the coefficients is the same (but with the powers of $q$ halved), and one just notes that here the characteristics $\big[{1/2 \atop 3/4}\big]$ and $\big[{1/2 \atop 7/4}\big]$ (where $t$ is 9 or 5) yield the $+$ sign, while with $\big[{1/2 \atop 1/4}\big]$ and $\big[{1/2 \atop 5/4}\big]$ (i.e., if $t$ is 11 or 7) the sign is $-$. For odd $N$ we write the coefficient $2\pi i\zeta_{8}^{tl}$ arising from a divisor $l$ of $N$ with $\frac{N}{l}\equiv1(\mathrm{mod\ 4})$ as the sum of $\pi i(\zeta_{8}^{tl}-\overline{\zeta}_{8}^{tl})$ and $\pi i(\zeta_{8}^{tl}+\overline{\zeta}_{8}^{tl})$, while for divisors with $\frac{N}{l}\equiv3(\mathrm{mod\ 4})$ the coefficient $-2\pi i\overline{\zeta}_{8}^{tl}$ equals the difference between $\pi i(\zeta_{8}^{tl}-\overline{\zeta}_{8}^{tl})$ and $\pi i(\zeta_{8}^{tl}+\overline{\zeta}_{8}^{tl})$. The terms involving $\zeta_{8}^{tl}-\overline{\zeta}_{8}^{tl}$ combine to $\frac{\pi i}{2}\cdot\sqrt{2}i\big(\frac{-2}{t}\big)s_{2}(N)$ by part $(ii)$ of Lemma \ref{recogser}, and since these expressions appear only for odd $N$, part $(ii)$ of Corollary \ref{pow2twist} shows that they produce the term containing $i$ in the rightmost brackets in all the four Equations \eqref{1/21/4}, \eqref{1/23/4}, \eqref{1/25/4}, and \eqref{1/27/4} (the sign is determined by the value $t$ such that $\delta=\frac{12-t}{4}$, and the variable $\tau$ still halved). On the other hand, in the terms with $\zeta_{8}^{tl}-\overline{\zeta}_{8}^{tl}$ we have the multiplier $\big(\frac{-1}{N/l}\big)$ (because of the sign depending on the residue of $\frac{N}{l}$ modulo 4), which we write as the product of $\big(\frac{-1}{N}\big)$ and $\big(\frac{-1}{l}\big)$. Then part $(ii)$ of Lemma \ref{recogser} identifies the resulting coefficient as $\frac{\pi i}{2}\cdot\sqrt{2}\big(\frac{2}{t}\big)\big(\frac{-1}{N}\big)s_{2}(N)$, which according to part $(iii)$ of Corollary \ref{pow2twist} produce the remaining terms in the four equations in question (again considering the sign depending on the value of $t$ and the halving of $\tau$). This completes the proof of the theorem.
\end{proof}
We note that our Equations \eqref{01/4} and \eqref{03/4} appear in Theorem 7.7 of \cite{[M3]} as well.

Some theta derivatives in which $\varepsilon$ is $\frac{1}{4}$ or $\frac{3}{4}$ can also be evaluated using this method:
\begin{thm}
The equalities
\begin{equation}
\theta\big[{\textstyle{1/4 \atop 1}}\big]'(\tau,0)=\tfrac{\pi i}{4}\theta\big[{\textstyle{1/4 \atop 1}}\big](\tau,0)\theta\big[{\textstyle{0 \atop 0}}\big]\big(\tfrac{\tau}{4},0\big)\Big(2\theta\big[{\textstyle{0 \atop 0}}\big]\big(\tfrac{\tau}{2},0\big)-\theta\big[{\textstyle{0 \atop 0}}\big]\big(\tfrac{\tau}{4},0\big)\Big) \label{1/41}
\end{equation}
and
\begin{equation}
\theta\big[{\textstyle{3/4 \atop 1}}\big]'(\tau,0)=\tfrac{\pi i}{4}\theta\big[{\textstyle{3/4 \atop 1}}\big](\tau,0)\theta\big[{\textstyle{0 \atop 0}}\big]\big(\tfrac{\tau}{4},0\big)\Big(2\theta\big[{\textstyle{0 \atop 0}}\big]\big(\tfrac{\tau}{2},0\big)+\theta\big[{\textstyle{0 \atop 0}}\big]\big(\tfrac{\tau}{4},0\big)\Big) \label{3/41}
\end{equation}
hold. Moreover, for $\theta\big[{1/4 \atop 0}\big]'(\tau,0)$ and $\theta\big[{3/4 \atop 0}\big]'(\tau,0)$ we have the respective expressions
\begin{equation}
\tfrac{\pi i}{4}\theta\big[{\textstyle{1/4 \atop 0}}\big](\tau,0)\Big(2\theta\big[{\textstyle{0 \atop 0}}\big]\big(\tfrac{\tau}{2},0\big)\theta\big[{\textstyle{0 \atop 1}}\big]\big(\tfrac{\tau}{4},0\big)-\theta\big[{\textstyle{0 \atop 0}}\big]^{2}\big(\tfrac{\tau}{2},0\big)+\theta\big[{\textstyle{1 \atop 0}}\big]^{2}\big(\tfrac{\tau}{2},0\big)\Big) \label{1/40}
\end{equation}
and
\begin{equation}
\tfrac{\pi i}{4}\theta\big[{\textstyle{3/4 \atop 0}}\big](\tau,0)\Big(2\theta\big[{\textstyle{0 \atop 0}}\big]\big(\tfrac{\tau}{2},0\big)\theta\big[{\textstyle{0 \atop 1}}\big]\big(\tfrac{\tau}{4},0\big)+\theta\big[{\textstyle{0 \atop 0}}\big]^{2}\big(\tfrac{\tau}{2},0\big)-\theta\big[{\textstyle{1 \atop 0}}\big]^{2}\big(\tfrac{\tau}{2},0\big)\Big). \label{3/40}
\end{equation}
Four additional theta derivatives are expressed by the formulae
\[\theta\big[{\textstyle{1/4 \atop 1/2}}\big]'(\tau,0)=\tfrac{\pi i}{4}\theta\big[{\textstyle{1/4 \atop 1/2}}\big](\tau,0)\Big[2\theta\big[{\textstyle{0 \atop 1}}\big]\big(\tfrac{\tau}{2},0\big)\Big(\theta\big[{\textstyle{0 \atop 0}}\big](\tau,0)-i\theta\big[{\textstyle{1 \atop 0}}\big](\tau,0)\Big)+\]
\begin{equation}
-\theta\big[{\textstyle{0 \atop 0}}\big]^{2}(\tau,0)+\theta\big[{\textstyle{1 \atop 0}}\big]^{2}(\tau,0\big)+i\theta\big[{\textstyle{1 \atop 0}}\big]^{2}\big(\tfrac{\tau}{2},0\big)\Big], \label{1/41/2}
\end{equation}
\[\theta\big[{\textstyle{1/4 \atop 3/2}}\big]'(\tau,0)=\tfrac{\pi i}{4}\theta\big[{\textstyle{1/4 \atop 3/2}}\big](\tau,0)\Big[2\theta\big[{\textstyle{0 \atop 1}}\big]\big(\tfrac{\tau}{2},0\big)\Big(\theta\big[{\textstyle{0 \atop 0}}\big](\tau,0)+i\theta\big[{\textstyle{1 \atop 0}}\big](\tau,0)\Big)+\]
\begin{equation}
-\theta\big[{\textstyle{0 \atop 0}}\big]^{2}(\tau,0)+\theta\big[{\textstyle{1 \atop 0}}\big]^{2}(\tau,0)-i\theta\big[{\textstyle{1 \atop 0}}\big]^{2}\big(\tfrac{\tau}{2},0\big)\Big], \label{1/43/2}
\end{equation}
\[\theta\big[{\textstyle{3/4 \atop 1/2}}\big]'(\tau,0)=\tfrac{\pi i}{4}\theta\big[{\textstyle{3/4 \atop 1/2}}\big](\tau,0)\Big[2\theta\big[{\textstyle{0 \atop 1}}\big]\big(\tfrac{\tau}{2},0\big)\Big(\theta\big[{\textstyle{0 \atop 0}}\big](\tau,0)+i\theta\big[{\textstyle{1 \atop 0}}\big](\tau,0)\Big)+\]
\begin{equation}
+\theta\big[{\textstyle{0 \atop 0}}\big]^{2}(\tau,0)-\theta\big[{\textstyle{1 \atop 0}}\big]^{2}(\tau,0)+i\theta\big[{\textstyle{1 \atop 0}}\big]^{2}\big(\tfrac{\tau}{2},0\big)\Big], \label{3/41/2}
\end{equation}
and \[\theta\big[{\textstyle{3/4 \atop 3/2}}\big]'(\tau,0)=\tfrac{\pi i}{4}\theta\big[{\textstyle{3/4 \atop 3/2}}\big](\tau,0)\Big[2\theta\big[{\textstyle{0 \atop 1}}\big]\big(\tfrac{\tau}{2},0\big)\Big(\theta\big[{\textstyle{0 \atop 0}}\big](\tau,0)-i\theta\big[{\textstyle{1 \atop 0}}\big](\tau,0)\Big)+\]
\begin{equation}
+\theta\big[{\textstyle{0 \atop 0}}\big]^{2}(\tau,0)-\theta\big[{\textstyle{1 \atop 0}}\big]^{2}(\tau,0)-i\theta\big[{\textstyle{1 \atop 0}}\big]^{2}\big(\tfrac{\tau}{2},0\big)\Big]. \label{3/43/2}
\end{equation} \label{4epsinZ}
\end{thm}

\begin{proof}
With $\varepsilon=\frac{1}{4}$ (resp. $\varepsilon=\frac{3}{4}$) and $\delta=1$ (i.e., $\beta=1$) Equation \eqref{logdergen} contains the constant coefficient $\frac{\pi i}{4}$ (resp. $\frac{3\pi i}{4}$), together with the terms $2\pi iq^{l(8n-5)/8}$ (resp. $2\pi iq^{l(8n-7)/8}$) and $-2\pi iq^{l(8n-3)/8}$ (resp. $-2\pi iq^{l(8n-1)/8}$) with positive natural $n$ and $l$. Since part $(iii)$ of Lemma \ref{recogser}, with the (odd) divisor $d$ being $\frac{N}{l}$, recognizes the resulting coefficient of $q^{N/8}$ as $\frac{\pi i}{2}s_{2}(N)-\frac{\pi i}{4}s_{1}(N)$ (resp. $\frac{\pi i}{2}s_{2}(N)+\frac{\pi i}{4}s_{1}(N)$), this proves Equations \eqref{1/41} and \eqref{3/41} (note that $\tau$ must be divided by 4 for such a power of $q$). Replacing the value of $\delta$ to be 0 (hence $\beta$ becomes $-1$) results in multiplying any term with index $l$ by $(-1)^{l}$, which is the same as $(-1)^{N}$ since $\frac{N}{l}$ is odd. Combining what we just evaluated with parts $(i)$ and $(ii)$ of Corollary \ref{pow2twist} then establishes Equations \eqref{1/40} and \eqref{3/40}.

We now turn to the remaining four theta derivatives, where $\delta$ is $\frac{1}{2}$ or $\frac{3}{2}$, hence $\beta=\pm i$. Hence in the series from Equation \eqref{logdergen}, any $q$-power involving $l$ is now multiplied by a sign times $2\pi i(\pm i)^{l}$. As in the proof of the Equations \eqref{1/21/4}, \eqref{1/23/4}, \eqref{1/25/4}, and \eqref{1/27/4} in Theorem \ref{4deltainZ}, we separate this series according to the parity of $l$ (or equivalently $N$), and observe that the coefficient for $q^{N/8}$ with even $N$ here becomes precisely the expression for the coefficient of $q^{(N/2)/8}$ in the evaluation that lead to Equations \eqref{1/40} and \eqref{3/40} (with the same $\varepsilon$). This yields, after the resulting duplication of the variable $\tau$, the terms not involving $i$ inside the parentheses in each of the Equations \eqref{1/41/2}, \eqref{1/43/2}, \eqref{3/41/2}, and \eqref{3/43/2}. For the terms with odd $l$ we note that the coefficient $-(\mp i)^{l}$ is the same as $(\pm i)^{l}=\pm i^{l}$, so that for odd $N$ we have to take the sum of $\pm2\pi i \cdot i^{l}$ over all the relevant divisors $l$ of $N$ (i.e., in which $\frac{N}{l}$ is congruent to 3 or to 5 modulo 8 if $\varepsilon=\frac{1}{4}$, and for which $\frac{N}{l}$ has residue 1 or 7 modulo 8 when $\varepsilon=\frac{3}{4}$). Since $i \cdot i^{l}$ equals $i \cdot i^{N}=-\big(\frac{-1}{N}\big)$ when $\frac{N}{l}\equiv1(\mathrm{mod\ }4)$ and $-i \cdot i^{N}=+\big(\frac{-1}{N}\big)$ if $\frac{N}{l}\equiv3(\mathrm{mod\ }4)$, and this depends on $l$ only through the residue of $\frac{N}{l}$ modulo 8, we may apply part $(iii)$ of Lemma \ref{recogser} with the divisors counted in $\sigma_{t+8\mathbb{Z}}(N)$ being $\frac{N}{l}$. This determines the coefficient of $q^{N/8}$ for odd $N$ as $\pm\frac{\pi}{4}\big(\frac{-1}{N}\big)\big(2s_{2}(N)-s_{1}(N)\big)$ if $\varepsilon=\frac{1}{4}$ and as $\mp\frac{\pi}{4}\big(\frac{-1}{N}\big)\big(2s_{2}(N)+s_{1}(N)\big)$ when $\varepsilon=\frac{3}{4}$. Part $(iii)$ of Corollary \ref{pow2twist}, with $\tau$ divided by 4 and with the sign corresponding to $\delta$, thus produces the remaining terms in Equations \eqref{1/41/2}, \eqref{1/43/2}, \eqref{3/41/2}, and \eqref{3/43/2}. This proves the theorem.
\end{proof}

In the proofs of Theorems \ref{2chinZ}, \ref{4deltainZ}, and \ref{4epsinZ} we have only stated the value of the constant term of the logarithmic derivative from Equation \eqref{logdereps1} or \eqref{logdergen}. It is easily verified, using the fact that $\theta\big[{0 \atop 0}\big]$ and $\theta\big[{0 \atop 1}\big]$ have constant term 1, $\theta\big[{1 \atop 0}\big]$ does not have a constant term, and all the powers of $q$ involved are non-negative, that the constant coefficient in each of the final formulae is indeed the asserted one. We conclude this section by noting that all the theta functions appearing in Theorem \ref{4epsinZ} with non-integral multiples of $\tau$ as their arguments can be replaced, via Equation \eqref{pardecom} (or sometimes Lemma 1.6 of Chapter 2 of \cite{[FK]}), by alternative combinations of theta functions in which only integral multiples of $\tau$ are substituted.

\section{Theta Derivatives with Characteristics from $\frac{1}{3}\mathbb{Z}$ \label{Charin1/3Z}}

\cite{[M1]} also quotes an identity from \cite{[F]} (numbered as Equation (2) in the former reference, as well as Equation (1.2) in the follow-up paper \cite{[M2]}) involving the theta derivative with characteristics $\big[{1 \atop 1/3}\big]$. For this theta derivatives, and some others, we can prove the following expressions.
\begin{thm}
With integral $\varepsilon$ and $\delta=\frac{1}{3}$ we have
\begin{equation}
\theta\big[{\textstyle{1 \atop 1/3}}\big]'(\tau,0)=-\tfrac{\pi}{\sqrt{3}}\theta\big[{\textstyle{1 \atop 1/3}}\big](\tau,0)\Theta_{3}(2\tau) \label{11/3}
\end{equation}
and
\begin{equation}
\theta\big[{\textstyle{0 \atop 1/3}}\big]'(\tau,0)=-\tfrac{\pi}{\sqrt{3}}\theta\big[{\textstyle{0 \atop 1/3}}\big](\tau,0)\Big(\Theta_{3}(\tau)-\Theta_{3}(2\tau)\Big), \label{01/3}
\end{equation}
while if $\delta$ is integral and $\varepsilon=\frac{1}{3}$ the theta derivatives are
\begin{equation}
\theta\big[{\textstyle{1/3 \atop 1}}\big]'(\tau,0)=\tfrac{\pi i}{3}\theta\big[{\textstyle{1/3 \atop 1}}\big](\tau,0)\Theta_{3}\big(\tfrac{2\tau}{3}\big) \label{1/31}
\end{equation}
and
\begin{equation}
\theta\big[{\textstyle{1/3 \atop 0}}\big]'(\tau,0)=\tfrac{\pi i}{3}\theta\big[{\textstyle{1/3 \atop 0}}\big](\tau,0)\Big(2\Theta_{3}\big(\tfrac{4\tau}{3}\big)-\Theta_{3}\big(\tfrac{2\tau}{3}\big)\Big). \label{1/30}
\end{equation}
Going over to characteristics in which $\delta=\frac{2}{3}$, we find that
\begin{equation}
\theta\big[{\textstyle{1 \atop 2/3}}\big]'(\tau,0)=-\tfrac{\pi}{\sqrt{3}}\theta\big[{\textstyle{1 \atop 2/3}}\big](\tau,0)\Big(2\Theta_{3}(4\tau)+\Theta_{3}(2\tau)\Big) \label{12/3}
\end{equation}
and
\begin{equation}
\theta\big[{\textstyle{0 \atop 2/3}}\big]'(\tau,0)=-\tfrac{\pi}{\sqrt{3}}\theta\big[{\textstyle{0 \atop 2/3}}\big](\tau,0)\Big(\Theta_{3}(\tau)+\Theta_{3}(2\tau)-2\Theta_{3}(4\tau)\Big). \label{02/3}
\end{equation}
On the other hand, when $\varepsilon=\frac{2}{3}$ we obtain
\begin{equation}
\theta\big[{\textstyle{2/3 \atop 1}}\big]'(\tau,0)=\tfrac{\pi i}{3}\theta\big[{\textstyle{2/3 \atop 1}}\big](\tau,0)\Big(\Theta_{3}\big(\tfrac{\tau}{3}\big)+\Theta_{3}\big(\tfrac{2\tau}{3}\big)\Big) \label{2/31}
\end{equation}
and
\begin{equation}
\theta\big[{\textstyle{2/3 \atop 0}}\big]'(\tau,0)=\tfrac{\pi i}{3}\theta\big[{\textstyle{2/3 \atop 0}}\big](\tau,0)\Big(2\Theta_{3}\big(\tfrac{4\tau}{3}\big)-\Theta_{3}\big(\tfrac{\tau}{3}\big)+\Theta_{3}\big(\tfrac{2\tau}{3}\big)\Big). \label{2/30}
\end{equation} \label{3ZandZ}
\end{thm}

\begin{proof}
Setting $\delta=\frac{1}{3}$ (hence $\beta=\zeta_{3}$) in Equation \eqref{logdereps1} yields the series with the constant term $\pi i\frac{1+\zeta_{3}}{1-\zeta_{3}}=-\frac{\pi}{\sqrt{3}}$ (recall that $1+\zeta_{3}=\zeta_{6}$ and $\Im\overline{\zeta}_{6}=-\frac{\sqrt{3}}{2}$) and in which the coefficient in front of $q^{N}$ is $2\pi i\sum_{l|N}(\zeta_{3}^{l}-\overline{\zeta}_{3}^{l})$. But this expression is $-\frac{\pi}{\sqrt{3}}s_{3}(N)$ by part $(iv)$ of Lemma \ref{recogser}, and Equation \eqref{11/3} directly follows. Replacing the value of $\varepsilon$ to be 0, and using Equation \eqref{logdergen}, we obtain that $q^{N/2}$ now comes with the coefficient $2\pi i\sum_{l|N,\ 2\not|N/l}(\zeta_{3}^{l}-\overline{\zeta}_{3}^{l})$. Part $(ii)$ of Corollary \ref{pow3twist} then establishes Equation \eqref{01/3}. Next, set $\varepsilon=\frac{1}{3}$ and $\delta=1$ (i.e., $\beta=1$) in Equation \eqref{logdergen}, and we get the constant $\frac{\pi i}{3}$ together with the terms $2\pi iq^{l(3n-2)/3}$ and $-2\pi iq^{l(3n-1)/3}$ for natural $n$ and $l$. Since in the resulting coefficient of $q^{N/3}$ the contribution arising from the divisor $d=\frac{N}{l}$ is easily seen to be $2\pi i\big(\frac{d}{3}\big)$, part $(iv)$ of Lemma \ref{recogser} (with $\tau$ rescaled) implies Equation \eqref{1/31}. Taking the value of $\delta$ to be 0 (so that $\beta=-1$), the expressions just mentioned above are multiplied by $(-1)^{l}$, which is $(-1)^{N/d}$ for $d=\frac{N}{l}$. Equation \eqref{1/30} thus follows from part $(iii)$ of Corollary \ref{pow3twist} (with the same rescaling).

We now substitute $\delta=\frac{2}{3}$ (which implies $\beta=\zeta_{6}$) in Equation \eqref{logdereps1}. The constant $\pi i\frac{1+\zeta_{6}}{1-\zeta_{6}}$ equals $-\sqrt{3}\pi$ (as $1-\zeta_{6}=\overline{\zeta}_{6}$ and $\zeta_{3}+\zeta_{6}=i\sqrt{3}$), and $q^{N}$ comes multiplied by $2\pi i\sum_{l|N}(\zeta_{6}^{l}-\overline{\zeta}_{6}^{l})$. Part $(iii)$ of Corollary \ref{pow3twist} (with the sign inverted and $\tau$ doubled) implies Equation \eqref{12/3}. Leaving the value of $\delta$ untouched, but working with $\varepsilon=0$ in Equation \eqref{logdergen}, we write $N=l(2n-1)$ as usual, and obtain that the coefficient of $q^{N/2}$ is $2\pi i\sum_{l|N,\ 2\not|N/l}(\zeta_{6}^{l}-\overline{\zeta}_{6}^{l})$. Part $(iv)$ of Corollary \ref{pow3twist} thus yields Equation \eqref{02/3}. Turning to the case where $\varepsilon=\frac{2}{3}$ and $\delta=\beta=1$, the series from Equation \eqref{logdergen} consists of the constant term $\frac{2\pi i}{3}$ as well as the terms $2\pi iq^{l(6n-5)/6}$ and $-2\pi iq^{l(6n-1)/6}$. As the divisor $d=\frac{N}{l}$ again contributes $2\pi i\big(\frac{d}{3}\big)$ to $q^{N/6}$ but is now restricted to be odd, part $(ii)$ of Corollary \ref{pow3twist} (with $\tau$ divided by 3) proves Equation \eqref{2/31}. Finally, with $\delta=0$ and $\beta=-1$ (and the same value of $\varepsilon$) we get the same expressions but multiplied by $(-1)^{l}$. As $d=\frac{N}{l}$ is odd here, this sign is the same as $(-1)^{N}$, and Equation \eqref{2/30} follows from part $(iv)$ of Corollary \ref{pow3twist} (again with the rescaled $\tau$). This completes the proof of the theorem.
\end{proof}

The arguments from the proof of Propositions \ref{intchar} and \ref{halfchar}, combined with product identities of the form $\prod_{n=1}^{\infty}(1 \pm q^{3n-1})(1 \pm q^{3n-2})=\prod_{n=1}^{\infty}\frac{1 \pm q^{n}}{1 \pm q^{3n}}$ and $\prod_{n=1}^{\infty}(1 \pm q^{6n-1})(1 \pm q^{6n-5})=\prod_{n=1}^{\infty}\frac{1 \pm q^{2n-1}}{1 \pm q^{6n-3}}$ together with equalities similar to $(1\pm\zeta_{3}q^{n})(1\pm\overline{\zeta}_{3}q^{n})=\frac{1 \pm q^{3n}}{1 \pm q^{n}}$, can be used for establishing the following result, which essentially contains all the formulae from Theorem 8.2 of \cite{[K]} and the weight $\frac{1}{2}$ part of Theorem 1.2 of \cite{[LO]} (with some additional formulae from Theorem 8.1 of the former reference as well as Theorem 1.1 of the latter).
\begin{prop}
The theta constants appearing in the equations from Theorem \ref{3ZandZ} have the following representations as single eta quotients:
\[\theta\big[{\textstyle{1 \atop 1/3}}\big](\tau,0)=\sqrt{3}\eta(3\tau),\quad\theta\big[{\textstyle{0 \atop 1/3}}\big](\tau,0)=\frac{\eta^{2}(\tau)\eta\big(\frac{3\tau}{2}\big)}{\eta\big(\frac{\tau}{2}\big)\eta(3\tau)},\quad\theta\big[{\textstyle{1/3 \atop 1}}\big](\tau,0)=\zeta_{12}\eta\big(\tfrac{\tau}{3}\big),\] \[\theta\big[{\textstyle{1/3 \atop 0}}\big](\tau,0)=\frac{\eta^{2}(\tau)\eta\big(\frac{2\tau}{3}\big)}{\eta(2\tau)\eta\big(\frac{\tau}{3}\big)},\quad\theta\big[{\textstyle{1 \atop 2/3}}\big](\tau,0)=\frac{\eta^{2}(\tau)\eta(6\tau)}{\eta(2\tau)\eta(3\tau)},\] \[\theta\big[{\textstyle{0 \atop 2/3}}\big](\tau,0)=\frac{\eta^{2}(3\tau)\eta(2\tau)\eta\big(\frac{\tau}{2}\big)}{\eta(\tau)\eta\big(\frac{3\tau}{2}\big)\eta(6\tau)},\quad\theta\big[{\textstyle{2/3 \atop 1}}\big](\tau,0)=\zeta_{6}\frac{\eta^{2}(\tau)\eta\big(\frac{\tau}{6}\big)}{\eta\big(\frac{\tau}{2}\big)\eta\big(\frac{\tau}{3}\big)},\] and \[\theta\big[{\textstyle{2/3 \atop 0}}\big](\tau,0)=\frac{\eta^{2}\big(\frac{\tau}{3}\big)\eta(2\tau)\eta\big(\frac{\tau}{2}\big)}{\eta(\tau)\eta\big(\frac{2\tau}{3}\big)\eta\big(\frac{\tau}{6}\big)}.\] \label{charover3}
\end{prop}

Once again, several additional identities can now be proved.
\begin{cor}
The theta constant identity $\theta\big[{1 \atop 1/3}\big](\tau,0)=\overline{\zeta}_{12}\sqrt{3}\theta\big[{1/3 \atop 1}\big](9\tau,0)$ holds, as well as the eta identities
\[\eta^{3}\big(\tfrac{4\tau}{9}\big)+3\eta^{3}(4\tau)=\frac{\eta^{6}\big(\frac{2\tau}{3}\big)\eta^{6}(2\tau)+4\eta^{2}\big(\frac{\tau}{3}\big)\eta^{4}\big(\frac{4\tau}{3}\big)\eta^{2}(\tau)\eta^{4}(4\tau)}
{\eta\big(\frac{\tau}{3}\big)\eta^{2}\big(\frac{4\tau}{3}\big)\eta^{2}(\tau)\eta^{2}(4\tau)\eta\big(\frac{2\tau}{3}\big)\eta(2\tau)}\] and
\[\eta^{3}(\tau)+9\eta^{3}(9\tau)=\frac{\eta^{6}(2\tau)\eta^{6}(6\tau)+4\eta^{2}(\tau)\eta^{4}(4\tau)\eta^{2}(3\tau)\eta^{4}(12\tau)}{\eta^{2}(\tau)\eta^{2}(4\tau)\eta(3\tau)\eta^{2}(12\tau)\eta(2\tau)\eta(6\tau)}.\] As a consequence we also get the equality \[\eta(\tau)\big[\eta^{3}\big(\tfrac{4\tau}{3}\big)+3\eta^{3}(12\tau)\big]=\eta(3\tau)[\eta^{3}(\tau)+9\eta^{3}(9\tau)].\] \label{Foureta1/3}
\end{cor}

\begin{proof}
The first identity follows immediately from Proposition \ref{charover3}. For the theta identities, note that the Fourier expansion of $\theta\big[{1/3 \atop 1}\big]'(\tau,0)$ yields, by differentiating Equation \eqref{defser}, the series $\frac{\pi i\zeta_{12}}{3}\sum_{3\not|n}n\big(\tfrac{-4}{n}\big)q^{n^{2}/18}$. By separating the series from Equation \eqref{eta3exp} (with $\tau$ rescaled) according to the divisibility of the summation index by 3 we recognize this as the series of $\frac{\pi i\zeta_{12}}{3}\big[\eta^{3}\big(\frac{4\tau}{9}\big)+3\eta^{3}(4\tau)\big]$. Similarly, for $\theta\big[{1 \atop 1/3}\big]'(\tau,0)$ we get, by the same means together with analyzing odd powers of $\zeta_{12}$, the series $-\pi\sum_{n=0}^{\infty}a_{n}n\big(\tfrac{-4}{n}\big)q^{n^{2}/8}$ with $a_{n}$ being $-2$ when 3 divides $n$ and $+1$ otherwise. A similar argument shows that this series represents $-\pi[\eta^{3}(\tau)+9\eta^{3}(9\tau)]$. We substitute these expressions into Equations \eqref{1/31} and \eqref{11/3}, eliminate the scalar ($\frac{\pi i\zeta_{12}}{3}$ or $-\pi$), apply the relevant formulae from Proposition \ref{charover3}, and note that Proposition \ref{intchar} implies $\Theta_{3}(\tau)$ equals $\frac{\eta^{5}(\tau)\eta^{5}(3\tau)}{\eta^{2}(\tau/2)\eta^{2}(2\tau)\eta^{2}(3\tau/2)\eta^{2}(6\tau)}+4\frac{\eta^{2}(2\tau)\eta^{2}(6\tau)}{\eta(\tau)\eta(3\tau)}$ (but rescale $\tau$ appropriately in the substitutions). This proves the first two eta identities, and the third one now follows in two ways: One may multiply $\tau$ by 3 in the second identity and normalize to get a common denominator, or one notes that $\Theta_{3}(2\tau)$ can be obtained either as $\frac{3}{\pi i}\theta\big[{1/3 \atop 1}\big]'(3\tau,0)\big/\theta\big[{1/3 \atop 1}\big](3\tau,0)$ from Equation \eqref{1/31} or as $-\frac{\sqrt{3}}{\pi}\theta\big[{1 \atop 1/3}\big]'(\tau,0)\big/\theta\big[{1 \atop 1/3}\big](\tau,0)$ via Equation \eqref{11/3}, and we can compare the expressions obtained above for these quotients. This proves the corollary.
\end{proof}
The theta constant identity from Corollary \ref{Foureta1/3} already appears implicitly in \cite{[F]}, in Equation (2) of \cite{[M1]}, and in Equation (1.2) of \cite{[M2]}. These expressions are also related to Corollaries 3.3, 3.4, and 3.5 of \cite{[M2]}. Proposition 3.6 and Corollary 3.7 of that reference, as well as the equalities appearing in Sections 5 and 6 of \cite{[M3]}, can be combined with our Equation \eqref{11/3} in Theorem \ref{3ZandZ} to deduce further theta and eta identities.

In relation to Theorem \ref{3ZandZ}, Proposition \ref{charover3}, and Corollary \ref{Foureta1/3}, we mention a classical identity of Ramanujan, appearing essentially with the reference $(v)$ in Entry 1 on Page 346 of \cite{[B1]}. By writing explicitly the function $\Theta_{3}$ using the ring $\mathbb{Z}[\zeta_{3}]$, this identity becomes
\[\Theta_{3}(2\tau)=\sum_{n,m\in\mathbb{Z}^{2}}q^{m^{2}-mn+n^{2}}=\frac{\eta^{3}\big(\frac{\tau}{3}\big)+3\eta^{3}(3\tau)}{\eta(\tau)},\] and it is equivalent to the representation of a theta derivative using a rational function of theta constants, namely
\[\theta\big[{\textstyle{1 \atop 1/3}}\big]'(\tau,0)=-\tfrac{\pi\zeta_{3}}{3\sqrt{3}}\theta\big[{\textstyle{1 \atop 1/3}}\big](\tau,0)\cdot\frac{3\theta\big[{\textstyle{1/3 \atop 1}}\big]^{3}(\tau,0)+\sqrt{3}i\theta\big[{\textstyle{1 \atop 1/3}}\big]^{3}(\tau,0)}{\theta\big[{\textstyle{1/3 \atop 1}}\big](3\tau,0)}.\]

\smallskip

The information that we have about $\Theta_{3}$ suffices for obtaining formulae for yet another 8 theta derivatives.
\begin{thm}
We have the equalities
\begin{equation}
\theta\big[{\textstyle{1/3 \atop 1/3}}\big]'(\tau,0)=\tfrac{\pi i}{3}\theta\big[{\textstyle{1/3 \atop 1/3}}\big](\tau,0)\Big(\zeta_{3}\Theta_{3}\big(\tfrac{2\tau}{3}\big)+\sqrt{3}i\overline{\zeta}_{3}\Theta_{3}(2\tau)\Big) \label{1/31/3}
\end{equation}
and
\begin{equation}
\theta\big[{\textstyle{1/3 \atop 5/3}}\big]'(\tau,0)=\tfrac{\pi i}{3}\theta\big[{\textstyle{1/3 \atop 5/3}}\big](\tau,0)\Big(\overline{\zeta}_{3}\Theta_{3}\big(\tfrac{2\tau}{3}\big)-\sqrt{3}i\zeta_{3}\Theta_{3}(2\tau)\Big), \label{1/35/3}
\end{equation}
while $\theta\big[{1/3 \atop 2/3}\big]'(\tau,0)$ and $\theta\big[{1/3 \atop 4/3}\big]'(\tau,0)$ are
\begin{equation}
\tfrac{\pi i}{3}\theta\big[{\textstyle{1/3 \atop 2/3}}\big](\tau,0)\Big(2\zeta_{3}\Theta_{3}\big(\tfrac{4\tau}{3}\big)-\overline{\zeta}_{3}\Theta_{3}\big(\tfrac{2\tau}{3}\big)
+2\sqrt{3}i\overline{\zeta}_{3}\Theta_{3}(4\tau)+\sqrt{3}i\zeta_{3}\Theta_{3}(2\tau)\Big) \label{1/32/3}
\end{equation}
and
\begin{equation}
\tfrac{\pi i}{3}\theta\big[{\textstyle{1/3 \atop 4/3}}\big](\tau,0)\Big(2\overline{\zeta}_{3}\Theta_{3}\big(\tfrac{4\tau}{3}\big)-\zeta_{3}\Theta_{3}\big(\tfrac{2\tau}{3}\big)
-2\sqrt{3}i\zeta_{3}\Theta_{3}(4\tau)-\sqrt{3}i\overline{\zeta}_{3}\Theta_{3}(2\tau)\Big) \label{1/34/3}
\end{equation}
respectively. On the other hand, the respective expressions for $\theta\big[{2/3 \atop 1/3}\big]'(\tau,0)$ and $\theta\big[{2/3 \atop 5/3}\big]'(\tau,0)$ are
\begin{equation}
\tfrac{\pi i}{3}\theta\big[{\textstyle{2/3 \atop 1/3}}\big](\tau,0)\Big(\zeta_{3}\Theta_{3}\big(\tfrac{2\tau}{3}\big)+\overline{\zeta}_{3}\Theta_{3}\big(\tfrac{\tau}{3}\big)
+\sqrt{3}i\overline{\zeta}_{3}\Theta_{3}(2\tau)-\sqrt{3}i\zeta_{3}\Theta_{3}(\tau)\Big) \label{2/31/3}
\end{equation}
and
\begin{equation}
\tfrac{\pi i}{3}\theta\big[{\textstyle{2/3 \atop 5/3}}\big](\tau,0)\Big(\overline{\zeta}_{3}\Theta_{3}\big(\tfrac{2\tau}{3}\big)+\zeta_{3}\Theta_{3}\big(\tfrac{\tau}{3}\big)
-\sqrt{3}i\zeta_{3}\Theta_{3}(2\tau)+\sqrt{3}i\overline{\zeta}_{3}\Theta_{3}(\tau)\Big), \label{2/35/3}
\end{equation}
and we have the theta derivatives
\[\theta\big[{\textstyle{2/3 \atop 2/3}}\big]'(\tau,0)=\frac{\pi i}{3}\theta\big[{\textstyle{2/3 \atop 2/3}}\big](\tau,0)\Big[2\zeta_{3}\Theta_{3}\big(\tfrac{4\tau}{3}\big)-\zeta_{3}\Theta_{3}\big(\tfrac{\tau}{3}\big)+\]
\begin{equation}
+\overline{\zeta}_{3}\Theta_{3}\big(\tfrac{2\tau}{3}\big)+2\sqrt{3}i\overline{\zeta}_{3}\Theta_{3}(4\tau)
-\sqrt{3}i\overline{\zeta}_{3}\Theta_{3}(\tau)-\sqrt{3}i\zeta_{3}\Theta_{3}(2\tau)\Big] \label{2/32/3}
\end{equation}
and
\[\theta\big[{\textstyle{2/3 \atop 4/3}}\big]'(\tau,0)=\frac{\pi i}{3}\theta\big[{\textstyle{2/3 \atop 4/3}}\big](\tau,0)\Big[2\overline{\zeta}_{3}\Theta_{3}\big(\tfrac{4\tau}{3}\big)-\overline{\zeta}_{3}\Theta_{3}\big(\tfrac{\tau}{3}\big)+\]
\begin{equation}
+\zeta_{3}\Theta_{3}\big(\tfrac{2\tau}{3}\big)-2\sqrt{3}i\zeta_{3}\Theta_{3}(4\tau)
+\sqrt{3}i\zeta_{3}\Theta_{3}(\tau)+\sqrt{3}i\overline{\zeta}_{3}\Theta_{3}(2\tau)\Big]. \label{2/34/3}
\end{equation} \label{3bothZ}
\end{thm}

\begin{proof}
We begin by setting $\varepsilon=\frac{1}{3}$ and $\delta=\frac{1}{3}$ (resp. $\delta=\frac{5}{3}$), so that $\beta=\zeta_{3}$ (resp. $\beta=\overline{\zeta}_{3}$) in Equation \eqref{logdergen}. The constant coefficient is $\frac{\pi i}{3}$, while the other terms are $2\pi i\zeta_{3}^{l}q^{l(3n-2)/3}$ (resp. $2\pi i\overline{\zeta}_{3}^{l}q^{l(3n-2)/3}$) and $-2\pi i\overline{\zeta}_{3}^{l}q^{l(3n-1)/3}$ (resp. $2\pi i\zeta_{3}^{l}q^{l(3n-1)/3}$). Writing the exponent as $\frac{N}{3}$, we find that $N$ is divisible by 3 if and only if $l$ is, and the proof of Equation \eqref{1/31} in Theorem \ref{3ZandZ} (with the variable $\tau$ here multiplied by 3) evaluates the sum arising from these terms as $\frac{\pi i}{3}\Theta_{3}(2\tau)$. For the rest, we decompose $2\pi i\zeta_{3}^{l}$ and $2\pi i\overline{\zeta}_{3}^{l}$ as the sum and difference of $\pi i(\zeta_{3}^{l}+\overline{\zeta}_{3}^{l})$ and $\pi i(\zeta_{3}^{l}-\overline{\zeta}_{3}^{l})$ respectively, and note that the first term here equals $-\pi i$ for every $l$ not divisible by 3. The contribution from the term of this sort that arises from a divisor $l$ is thus $-\pi i$ when $\frac{N}{l}\equiv1(\mathrm{mod\ }3)$ and $+\pi i$ if $\frac{N}{l}\equiv2(\mathrm{mod\ }3)$ for both our values of $\delta$ and $\beta$, so that part $(iv)$ of Lemma \ref{recogser} (or Proposition \ref{theta1} directly) with $d=\frac{N}{l}$ recognizes this contribution as $-\frac{\pi i}{6}s_{3}(N)$. The other terms amount to $\pm\pi i\sum_{l|N}(\zeta_{3}^{l}-\overline{\zeta}_{3}^{l})$ (since $N$ is not divisible by 3, the condition that 3 does not divide the quotient $\frac{N}{l}$, which is either $3n-2$ or $3n-1$, is redundant), with the $+$ appearing for $\delta=\frac{1}{3}$ and the $-$ when $\delta=\frac{5}{3}$. But this is just $\mp\frac{\pi}{2\sqrt{3}}s_{3}(N)$ by another application of part $(iv)$ of Lemma \ref{recogser} (now with $d=l$). Now, the sum of the coefficients multiplying $s_{3}(N)$ here is $\frac{\pi i}{3}\zeta_{3}$ (resp. $\frac{\pi i}{3}\overline{\zeta}_{3}$), and they appear only for $N$ not divisible by 3. Therefore part $(i)$ of Corollary \ref{pow3twist} (with $\tau$ rescaled) establishes both Equations \eqref{1/31/3} and \eqref{1/35/3}, where for the coefficient of $\Theta_{3}(2\tau)$ we have used the equality $1-\zeta_{3}=\sqrt{3}i\overline{\zeta}_{3}$ (resp. $1-\overline{\zeta}_{3}=-\sqrt{3}i\zeta_{3}$).

Still in the case where $\varepsilon=\frac{1}{3}$, but now with $\delta=\frac{2}{3}$ (resp. $\delta=\frac{4}{3}$), we get that $\beta=\zeta_{6}$ (resp. $\beta=\overline{\zeta}_{6}$). Hence every $\zeta_{3}$ in the above paragraph has to be replaced by $\zeta_{6}$. The argument proving Equation \eqref{1/30} in Theorem \ref{3ZandZ} (with $\tau$ multiplied by 3 again) then shows that the terms involving $q^{N/3}$ with $N$ divisible by 3 yield $\frac{\pi i}{3}\big(2\Theta_{3}(4\tau)-\Theta_{3}(2\tau)\big)$, and we again consider, for $N$ (or equivalently $l$) not divisibile by 3, the contributions arising from the terms $\pi i(\zeta_{6}^{l}+\overline{\zeta}_{6}^{l})$ and $\pi i(\zeta_{6}^{l}-\overline{\zeta}_{6}^{l})$. As the former equals the same value as $\pi i(\zeta_{3}^{l}+\overline{\zeta}_{3}^{l})$ for even $l$ but its additive inverse for odd $l$, we recognize the resulting coefficient as $-\frac{\pi i}{6}s_{3}^{\hat{s}}(N)$ from part $(iii)$ of Corollary \ref{pow3twist} (with $d=\frac{N}{l}$). But by the same part of the same corollary the remaining terms combine to $\pm\frac{\pi}{2\sqrt{3}}s_{3}^{s}(N)$ (again, with the $+$ when $\delta=\frac{2}{3}$ and the $-$ if $\delta=\frac{4}{3}$---note the different sign in the multiplier of $s_{3}^{s}(N)$ there), and recalling that only terms in which 3 does not divide $N$ are considered, we find that combining the relevant series from part $(v)$ of Corollary \ref{pow3twist} (with $\tau$ rescaled) yields $\frac{\pi i}{3}\big[2\zeta_{3}\Theta_{3}\big(\frac{4\tau}{3}\big)-2\zeta_{3}\Theta_{3}(4\tau)+\overline{\zeta}_{3}\Theta_{3}(2\tau)-\overline{\zeta}_{3}\Theta_{3}\big(\frac{2\tau}{3}\big)\big]$ (resp. $\frac{\pi i}{3}\big[2\overline{\zeta}_{3}\Theta_{3}\big(\frac{4\tau}{3}\big)-2\overline{\zeta}_{3}\Theta_{3}(4\tau)+\zeta_{3}\Theta_{3}(2\tau)-\zeta_{3}\Theta_{3}\big(\frac{2\tau}{3}\big)\big]$), with the same considerations about the coefficients. Adding the previous expression from the terms with $N$ divisible by 3 yields Equations \eqref{1/32/3} and \eqref{1/34/3}, using the equalities from the end of the previous paragraph.

Back in the case where $\delta$ is $\frac{1}{3}$ (resp. $\frac{5}{3}$) and $\beta$ is $\zeta_{3}$ (resp. $\overline{\zeta}_{3}$), but now with $\varepsilon=\frac{2}{3}$, the constant term becomes $\frac{2\pi i}{3}$, and in the non-constant terms (with $\zeta_{3}^{l}$ and $-\overline{\zeta}_{3}^{l}$ once more) the exponents $\frac{3n-2}{3}$ and $\frac{3n-1}{3}$ are replaced by $\frac{6n-5}{6}$ and $\frac{6n-1}{6}$ respectively. As this means halving the value of $\tau$ (in comparison to the case with $\varepsilon=\frac{1}{3}$) but taking only divisors $l$ with odd $\frac{N}{l}$, the proof of Equation \eqref{2/31} in Theorem \ref{3ZandZ} (with $\tau$ rescaled yet again) determines the series involving the expressions $q^{N/6}$ where 3 divides $N$ as $\frac{\pi i}{3}\big(\Theta_{3}(\tau)+\Theta_{3}(2\tau)\big)$. For the other terms we express again the coefficients in terms of $\pi i(\zeta_{3}^{l}+\overline{\zeta}_{3}^{l})$ and $\pi i(\zeta_{3}^{l}-\overline{\zeta}_{3}^{l})$, and the former terms contribute $-\pi i$ or $+\pi i$ according to the same rule of the residue of $\frac{N}{l}$ modulo 3, but now $\frac{N}{l}$ must also be odd. By part $(ii)$ of Corollary \ref{pow3twist} (with $d=\frac{N}{l}$) the resulting coefficient is $-\frac{\pi i}{6}s_{3}^{o}(N)$, and an application of the same part of that corollary with $d=l$ produces $\mp\frac{\pi}{2\sqrt{3}}s_{3}^{\hat{o}}(N)$ (where we take $-$ for $\delta=\frac{1}{3}$ and $+$ in case $\delta=\frac{5}{3}$). The resulting combination of the appropriate series from part $(v)$ of Corollary \ref{pow3twist}, with $\tau$ divided by 3, thus yields $\frac{\pi i}{3}\big[\zeta_{3}\Theta_{3}\big(\frac{2\tau}{3}\big)-\zeta_{3}\Theta_{3}(2\tau)+\overline{\zeta}_{3}\Theta_{3}\big(\frac{\tau}{3}\big)-\overline{\zeta}_{3}\Theta_{3}(\tau)\big]$ (resp. $\frac{\pi i}{3}\big[\overline{\zeta}_{3}\Theta_{3}\big(\frac{2\tau}{3}\big)-\overline{\zeta}_{3}\Theta_{3}(2\tau)+\zeta_{3}\Theta_{3}\big(\frac{\tau}{3}\big)-\zeta_{3}\Theta_{3}(\tau)\big]$), and putting in the contribution from the terms in which $N$ is divisible by 3 gives, using the usual equalities for the coefficients, the desired Equations \eqref{2/31/3} and \eqref{2/35/3}.

Finally, let the index $\varepsilon$ remain $\frac{2}{3}$, take $\delta=\frac{2}{3}$ (resp. $\delta=\frac{4}{3}$) and $\beta=\zeta_{6}$ (resp. $\delta=\overline{\zeta}_{6}$) again, and replace every $\zeta_{3}$ by $\zeta_{6}$ in the previous paragraph. The series concerning the terms $q^{N/6}$ where $N$ is divisible by 3 is evaluated by the proof of Equation \eqref{2/30} in Theorem \ref{3ZandZ} (with $\tau$ rescaled appropriately) as $\frac{\pi i}{3}\big(2\Theta_{3}(4\tau)-\Theta_{3}(\tau)+\Theta_{3}(2\tau)\big)$. Next, in the decomposition of the remaining terms using $\pi i(\zeta_{6}^{l}+\overline{\zeta}_{6}^{l})$ and $\pi i(\zeta_{6}^{l}-\overline{\zeta}_{6}^{l})$, the former expression amounts to $-(-1)^{l}\pi i$ or $+(-1)^{l}\pi i$ according to the residue of $\frac{N}{l}$ modulo 3 and under the restriction that the latter number is odd. Part $(iv)$ of Corollary \ref{pow3twist} (with $d=\frac{N}{l}$) writes the latter coefficient as $-\frac{\pi i}{6}(-1)^{N}s_{3}^{o}(N)$, and evaluates the contribution of the other terms as $\pm\frac{\pi}{2\sqrt{3}}(-1)^{N}s_{3}^{\hat{o}}(N)$ (where $+$ appears for $\delta=\frac{2}{3}$ and $-$ when $\delta=\frac{4}{3}$). From the last two equalities in part $(v)$ of Corollary \ref{pow3twist} (with the rescaling of $\tau$ as always) we deduce that the series involving only $q^{N/6}$ in which 3 does not divide $N$ equals $\frac{\pi i}{3}\big[2\zeta_{3}\Theta_{3}\big(\frac{4\tau}{3}\big)-\zeta_{3}\Theta_{3}\big(\frac{\tau}{3}\big)-2\zeta_{3}\Theta_{3}(4\tau)+\zeta_{3}\Theta_{3}(\tau)
+\overline{\zeta}_{3}\Theta_{3}\big(\frac{2\tau}{3}\big)-\overline{\zeta}_{3}\Theta_{3}(2\tau)\big]$ (resp. $\frac{\pi i}{3}\big[2\overline{\zeta}_{3}\Theta_{3}\big(\frac{4\tau}{3}\big)-\overline{\zeta}_{3}\Theta_{3}\big(\frac{\tau}{3}\big)-
2\overline{\zeta}_{3}\Theta_{3}(4\tau)+\overline{\zeta}_{3}\Theta_{3}(\tau)+\zeta_{3}\Theta_{3}\big(\frac{2\tau}{3}\big)-\zeta_{3}\Theta_{3}(2\tau)\big]$) in total, and adding the function arising from the terms with $q^{N/6}$ where 3 divides $N$ yields, after the usual coefficient manipulation, the required Equations \eqref{2/32/3} and \eqref{2/34/3}. This completes the proof of the theorem.
\end{proof}

Also here the constant terms in all the formulae appearing in Theorems \ref{3ZandZ} and \ref{3bothZ} are indeed the ones asserted in the proof (since $\Theta_{3}$ has constant term 1), using some evaluations involving $\zeta_{3}$ for the latter theorem.

\smallskip

We conclude by remarking that for characteristics with 5 in the denominator the description of the associated theta function may not be so straightforward as in Proposition \ref{theta1}, since the class number of $\mathbb{Q}(\sqrt{-5})$ is not 1. On the other hand, expressions involving rational functions of theta constants exist for such characteristics: See Theorem 5.2 and 6.2 of \cite{[M2]} for the case where $\varepsilon=1$ (and $\delta\in\frac{1}{5}\mathbb{Z}$), and Section 4 of \cite{[M3]} in general. Turning to 7 in the denominator, $\mathbb{Q}(\sqrt{-7})$ is of class number 1, and a normalization for $\Theta_{7}$ is the function appearing in Lemma 3.2 of \cite{[HL]} and denoted by $z_{7}$ in Section 3.2 of \cite{[O]} or in Section 7.6 of \cite{[C]} . However, since sums of 7th roots of unity with characters or restricting to only 2 residues modulo 7 (with opposite signs) does not immediately give the coefficients of $\Theta_{7}$, the evaluation with characteristics having this denominator is again substantially more difficult. It is possible though that allowing 6 in the denominator, or mixing 2 and 3 in the denominators of the two characteristics, may still yield expressions that are possible to evaluate in our methods. These questions are left fur future research.

\noindent\textsc{Einstein Institute of Mathematics, the Hebrew University of Jerusalem, Edmund Safra Campus, Jerusalem 91904, Israel}

\noindent E-mail address: zemels@math.huji.ac.il

\end{document}